\documentclass[11pt]{article}
\usepackage[margin=1in]{geometry}

\RequirePackage{graphicx}
\RequirePackage{amsmath}
\RequirePackage{amsfonts}
\RequirePackage{amssymb}
\RequirePackage{amsthm}
\ifdefined\intercal\else
  \DeclareMathSymbol{\intercal}{\mathbin}{AMSa}{"7C}
\fi
\RequirePackage{enumitem}
\RequirePackage{multirow}
\RequirePackage{booktabs}
\RequirePackage{threeparttable}
\RequirePackage{algorithm}
\RequirePackage{algpseudocode}
\RequirePackage{xcolor}
\RequirePackage{extarrows}
\RequirePackage[authoryear]{natbib}
\RequirePackage[colorlinks=true,citecolor=blue,urlcolor=blue,linkcolor=blue]{hyperref}

\makeatletter
\newcommand{\myitem}[1]{\item[#1]\protected@edef\@currentlabel{#1}}
\makeatother

\newcommand{\bfP}{\mathbf{P}}

\newcommand{\R}{\mathbb{R}}
\newcommand{\lb}{\left(}
\newcommand{\rb}{\right)}
\newcommand{\Var}{\operatorname{Var}}

\newcommand{\E}{\mathbb{E}}

\newcommand{\bfSigmahat}{	\hat{ \boldsymbol{\Sigma} }}

\newcommand{\bfQ}{\mathbf{Q}}
\newcommand{\bfSigma}{\boldsymbol{\Sigma}}
\newcommand{\bfLambda}{\boldsymbol{\Lambda}}
\newcommand{\bfX}{\mathbf{X}}
\newcommand{\bfI}{\mathbf{I}}

\newcommand{\bfy}{\mathbf{y}}
\newcommand{\op}{o_{\PR}(1)}

\newcommand{\cond}{\stackrel{\mathcal{D}}{\to}}
\newcommand{\conp}{\stackrel{\mathbb{P}}{\to}}
\newcommand{\tr}{\operatorname{tr}}

\newcommand{\PR}{\mathbb{P}}

\newcommand{\bfx}{\mathbf{x}}

\newcommand{\bfA}{\mathbf{A}}
\newcommand{\bfB}{\mathbf{B}}
\newcommand{\bfC}{\mathbf{C}}
\newcommand{\bfD}{\mathbf{D}}

\newcommand{\bfS}{\mathbf{S}}
\newcommand{\bfGamma}{\boldsymbol{\Gamma}}
\newcommand{\bfbeta}{\boldsymbol{\beta}}

\newcommand{\bfq}{\mathbf{q}}
\newcommand{\bfz}{\mathbf{z}}
\newcommand{\bfV}{\mathbf{V}}
\newcommand{\bfU}{\mathbf{U}}
\newcommand{\bfG}{\mathbf{G}}

\theoremstyle{plain}
\newtheorem{theorem}{Theorem}[section]
\newtheorem{lemma}{Lemma}[section]
\newtheorem{proposition}{Proposition}[section]

\theoremstyle{definition}

\newtheorem{example}{Example}[section]
\newtheorem{remark}{Remark}[section]
\newtheorem{assumption}{Assumption}[section]

\numberwithin{equation}{section}

\allowdisplaybreaks

\begin{document}

\title{Goodness-of-Fit Testing for Groupwise\\ Spherical Error Structures}

\author{
Daria~Tieplova\textsuperscript{1,2},
Nina~D\"ornemann\textsuperscript{1} and
Tim~Kutta\textsuperscript{1}\\
{\small \textsuperscript{1}Department of Mathematics, Aarhus University, Denmark}\\
{\small \textsuperscript{2}B. Verkin Institute for Low Temperature Physics and Engineering}\\
{\small of the NAS of Ukraine, Kharkiv, Ukraine}
}

\date{}
\maketitle

\begin{abstract}
The analysis of large data panels is important in econometrics and beyond. 
Prediction and inference methods for such data typically rely on simplifying model 
assumptions for the covariance structure of errors. One convenient assumption 
is what we call groupwise sphericity: that errors are uncorrelated across individuals and have 
constant variance within certain groups. While theoretically useful, in large panels 
groupwise sphericity is often too restrictive to apply in practice. We therefore develop 
new quantitative inference tools to test whether deviations from this model assumption are practically relevant. Our 
approach covers both large-dimensional data matrices and regression panels, in 
a regime where the cross-sectional dimension is proportional to the sample 
size. The theory is based on the analysis of extreme eigenvalues of the empirical covariance matrix and uses 
recent advances in random matrix theory. Numerical experiments demonstrate accurate size control and good power in finite samples.
\end{abstract}



\section{Introduction}

Many statistical procedures for analyzing panel data rely on simplifying assumptions on the
covariance structure of model errors (e.g., \cite{stock:watson:2008,baltagi:feng:kao:2011}). One useful simplification is to place units into known groups and assume cross-sectionally uncorrelated errors and within-group variance homogeneity. Such assumptions of a structured covariance make prediction and inference
more tractable, but in large panels are rarely expected to hold exactly.
This raises a natural goodness-of-fit question: is the assumed covariance
structure a sufficiently accurate approximation, and can the size of its
deviation from the truth be quantified?\\
In this paper, we study this question for data matrices and regression panels
with large cross-sectional dimension. To introduce the problem, we focus on the
simpler case of data matrices. The regression problem is treated later. For now, define
\begin{align} \label{e:Y}
    \mathbf{Y}=(\mathbf{y}_1,\ldots,\mathbf{y}_n)\in \mathbb{R}^{p\times n},
\end{align}
where \(n\) denotes the sample size and \(p\) the cross-sectional dimension.
We assume that the columns \(\mathbf{y}_1,\ldots,\mathbf{y}_n\) are i.i.d. with
mean vector \(\mu\in\mathbb{R}^p\) and cross-sectional covariance matrix
\(\bfSigma\in\mathbb{R}^{p\times p}\).
Next, suppose that the \(p\) components are partitioned into \(L\) known groups. After
a possible reordering of the coordinates, these groups may be written as
\begin{align}     \{1,\ldots,\ell_1\},\quad
    \{\ell_1+1,\ldots,\ell_1+\ell_2\},\quad \ldots,\quad
    \Bigl\{\sum_{i=1}^{L-1}\ell_i+1,\ldots,\sum_{i=1}^{L}\ell_i\Bigr\},
\end{align}
where \(\sum_{i=1}^{L}\ell_i=p\). The assumption that the variances are allowed
to differ across these groups, but are constant within each group, is closely
related to what is often called groupwise heteroskedasticity. In the present
paper, we consider the stronger covariance model which, in addition, imposes
the absence of cross-sectional correlations. We refer to this model as
\emph{groupwise sphericity}. Thus, \(\bfSigma\) satisfies groupwise sphericity
if
\[
    \bfSigma = \bfD,
\]
where
\begin{align} \label{e:def:gr:sph}
    \bfD
    =
    \operatorname{diag}\bigl(
    \sigma_1^2 \bfI_{\ell_1},
    \sigma_2^2 \bfI_{\ell_2},
    \ldots,
    \sigma_L^2 \bfI_{\ell_L}
    \bigr)
\end{align}
for some \(\sigma_1^2,\ldots,\sigma_L^2>0\), and $\bfI_{\ell}$ denotes the $\ell$-dimensional identity matrix. The special case \(L=1\)
corresponds to the classical assumption of sphericity.
Our aim is not only to test the exact model \(\bfSigma=\bfD\), but to quantify
how far a general covariance matrix is from this idealized structure. For a
general covariance matrix \(\bfSigma\), denote by
\(\bfSigma_{ii}\in\mathbb{R}^{\ell_i\times \ell_i}\) its \(i\)-th diagonal
block and define
\begin{align} \label{e:mat:block}
    \sigma_i^2 := \frac{1}{\ell_i}\operatorname{tr}(\bfSigma_{ii})>0,
    \qquad i=1,\ldots,L.
\end{align}
Let \(\bfD=\bfD(\bfSigma)\) be the corresponding groupwise spherical matrix
defined by \eqref{e:def:gr:sph}. We measure deviations from groupwise
sphericity by invoking the measure
\begin{align} \label{e:S:M}
    \mathcal{S}(\bfSigma)
    :=
    \bigl\|
        \bfD^{-1/2}\bfSigma \bfD^{-1/2} - \bfI_p
    \bigr\|,
\end{align}
with \(\|\cdot\|\) denoting the spectral norm. This quantity is scale-free and
satisfies the relation
\[
    \mathcal{S}(\bfSigma)=0
    \quad \Longleftrightarrow \quad
    \bfSigma \text{ is groupwise spherical}.
\]
More generally, \(\mathcal{S}(\bfSigma)\) measures the largest spectral
deviation from the groupwise spherical approximation. It therefore leads
naturally to relevant goodness-of-fit hypotheses of the form
\begin{align} \label{e:H0H1}
    H_0:\mathcal{S}(\bfSigma)\leq \Delta
    \qquad \text{versus} \qquad
    H_1:\mathcal{S}(\bfSigma)>\Delta,
\end{align}
where \(\Delta> 0\) is a user-specified tolerance level \citep[see also][]{lavergne:2014}. The aim of this paper is to test the hypothesis pair \eqref{e:H0H1}, or, equivalently, to produce confidence intervals for the quantity $\mathcal{S}(\bfSigma)$.
In scenarios where \(p\) is fixed and \(n\to\infty\), inference for
\(\mathcal{S}(\bfSigma)\) can be developed by classical arguments. In
particular, one may combine the delta method with differentiability properties
of the map
\[
    \bfSigma \mapsto \bfD(\bfSigma)^{-1/2}
        \bfSigma
        \bfD(\bfSigma)^{-1/2}.
\]

\noindent This fixed-dimensional regime, however, is not appropriate for the
settings considered here, where cross-sections are large. Notice that in many applications the cross-sectional dimension
is of the same order as the sample size. We therefore study the proportional
asymptotic regime in which \(p\) and \(n\) tend to infinity together.
In this regime, the problem changes fundamentally. The sample
covariance matrix $\widehat\bfSigma_n$ is not a consistent estimator of \(\bfSigma\) in spectral
norm. Consequently, the plug-in statistic
\(\mathcal{S}(\widehat\bfSigma_n)\) is not, in general, a consistent estimator
of \(\mathcal{S}(\bfSigma)\) \citep{bai:silverstein:2010}. Instead, its behavior is governed by the extreme
eigenvalue distribution of the high-dimensional random matrix $\widehat\bfSigma_n$.  In particular, small
deviations from groupwise sphericity, say $\mathcal{S}(\bfSigma)=\delta\approx0$, may be undetectable by \(\mathcal{S}(\widehat\bfSigma_n)\) because the limit of this statistic is independent of $\delta$ for all $\delta \in [0,\delta_0)$ and some value $\delta_0>0$ \citep{Baik:BenArous:Peche:2005, Ding:Yang:2021}.\footnote{For spiked alternatives, likelihood-ratio tests using the full spectrum can nevertheless have nontrivial asymptotic power in the subcritical regime \citep{onatski:moreira:hallin:2013}.} Since such minor deviations are also practically less relevant, testing for exactly $\mathcal{S}(\bfSigma)=0$ may not be particularly meaningful.
In contrast, larger deviations from groupwise sphericity are both practically meaningful and statistically detectable. In this scenario, the leading eigenvalues of $\widehat\bfSigma_n$ become an interpretable statistic that (after some bias correction) can be used to consistently estimate $\mathcal{S}(\bfSigma)$. \\
The contributions of this paper are as follows. First, we introduce a
scale-free spectral-norm measure of deviation from groupwise sphericity and use
it to formulate relevant goodness-of-fit hypotheses. Second, we develop
 inference methods for this measure in the regime where cross-sectional dimension and sample size are proportional,
based on the asymptotic behavior of extreme eigenvalues. Third, we extend the
methodology from i.i.d. high-dimensional vectors to linear regression
panels, where the covariance structure concerns the regression errors.\\
The remainder of the paper is organized as follows. In the rest of this
introduction, we discuss related literature. In Section \ref{sec_tests}, we
develop our main methodology and theory for \(\mathcal{S}(\bfSigma)\). We first consider the case of
i.i.d. high-dimensional observations as in \eqref{e:Y}, and then turn to linear
regression panels. Numerical results are presented in Section \ref{sec:exp}.
Proofs of the main theorems and further technical details are given at the end of the paper.

\subsection{Related literature}

We refer to a selection of related works, without attempting an exhaustive
review. The use of a positive tolerance in \eqref{e:H0H1} is closely related
to the literature on relevant hypotheses and equivalence testing, where the
size of a deviation determines its practical relevance. An econometric
treatment of approximate model restrictions is given by
\cite{lavergne:2014}. Recent applications to panel data include testing for
equivalence of pre-trends in difference-in-differences estimation
\citep{dette:schumann:2024} and validating approximate slope homogeneity
\citep{kutta:dette:2024}. Equivalence tests seek evidence that deviations
are sufficiently small; our hypotheses \eqref{e:H0H1} instead concern
whether deviations exceed the prescribed tolerance. Relevant hypotheses have also been considered in high-dimensional settings; see, for example, \cite{Bastian:Dette:Heiny:2024} for testing whether pairwise dependence measures exceed a prescribed threshold. \\
High-dimensional tests of exact sphericity have been studied by
\cite{ledoit:wolf:2002, chen:zhang:zhong:2010, hu:li:liu:wang:2019, zou2014multivariate, dornemann2024linear}, among many others.
For tests of sphericity of regression errors in fixed effects panels,
including non-normal disturbances, see \cite{baltagi:kao:peng:2015}.
\cite{onatski:moreira:hallin:2013} analyze the asymptotic power of
sphericity tests against spiked alternatives and clarify the role of the
phase transition for tests based on the largest sample eigenvalue. More recently, \cite{Zhang:Jiang:Yao:2025} develop a bootstrap approach to operator-norm-based tests of a specified covariance matrix.
\\
Our analysis builds on random matrix theory for spiked covariance
matrices \citep{Johnstone:2001}. \cite{bai:yao:2012, bai:yao:2008, Paul:2007, bai:ding:2012, Doernemann:Paul:2026} establish central limit theorems for sample eigenvalues in
such models, while
\cite{zhang:zheng:pan:zhong:2022, li2020asymptotic} derive central limit theorems for
spiked sample eigenvalues under general population covariance structures.
Related eigenvalue-based methods for comparing two covariance matrices
are developed in \cite{zhang:zheng:pan:zhong:2022} and
\cite{lam:doernemann:dette:2025}. Working in the supercritical regime,
we use corrected extreme eigenvalues to obtain inference for
$\mathcal{S}(\bfSigma)$, account for estimation of the group variances,
and extend the procedure to regression residuals.
\section{Quantifying groupwise
sphericity} \label{sec_tests}

We develop tests for the model assumption of approximate sphericity, beginning with a data matrix and subsequently extending our method to linear regression panels.

\subsection{Methodology}

Recall the setting from our introduction, and in particular, the definition of the $p$-dimensional data vectors $\mathbf{y}_i, 1 \le i \le n$ in \eqref{e:Y}, the group structure \eqref{e:def:gr:sph}, the matrix $\mathbf
D=\mathbf{D}(\bfSigma)$ in \eqref{e:def:gr:sph}, submatrices $\bfSigma_{ii}$ in \eqref{e:mat:block} and the goodness-of-fit measure $\mathcal{S}(\bfSigma)$ in \eqref{e:S:M}. We also define the empirical covariance matrix of the data 
\begin{align*}
    \widehat{\bfSigma}_n := \frac{1}{n} \sum_{j=1}^n \lb \mathbf{y}_j - \overline{\mathbf{y}} \rb \lb \mathbf{y}_j - \overline{\mathbf{y}}  \rb^\top , \quad\,\,\, \textnormal{where}\quad\,\,\, \overline{\mathbf{y}} := \frac{1}{n} \sum_{j=1}^n \mathbf{y}_j.
\end{align*} 
As a first step towards developing our new methodology, we briefly analyze the deviation measure $\mathcal{S}(\bfSigma)$. For this purpose, define for a symmetric matrix with real components $\mathbf{A} \in \mathbb{R}^{p \times p}$ the largest eigenvalue $\lambda_1(\mathbf{A})$ and the smallest eigenvalue $\lambda_p(\mathbf{A})$. Since $\bfD^{-1/2}\bfSigma \bfD^{-1/2}$ is symmetric and positive semidefinite, we can express the spectral norm in terms of the extreme eigenvalues
\begin{align} \label{e:simp:1}
      \bigl\|
        \bfD^{-1/2}\bfSigma \bfD^{-1/2} - \bfI_p
    \bigr\| = \max\big(|\lambda_1(\bfD^{-1/2}\bfSigma \bfD^{-1/2})-1|, |1-\lambda_p(\bfD^{-1/2}\bfSigma \bfD^{-1/2})|\big).
\end{align}
Here, we have used the fact that the identity matrix is simultaneously diagonalizable with any symmetric matrix.
Next, it is possible to eliminate the absolute values on the right-hand side of \eqref{e:simp:1}. To see this, consider the trace of $\bfD^{-1/2}\bfSigma \bfD^{-1/2}$, which is
\begin{align*}
    \operatorname{tr}[\bfD^{-1}\bfSigma ] = \sum_{i=1}^L \frac{1}{\sigma_i^2} \operatorname{tr}(\bfSigma_{ii}) = \sum_{i=1}^L \ell_i=p.
\end{align*}
Since all eigenvalues of $\bfD^{-1/2}\bfSigma \bfD^{-1/2}$ are non-negative, this implies that 
\[
\lambda_1(\bfD^{-1/2}\bfSigma \bfD^{-1/2})\ge 1 \quad \,\, \textnormal{and}\quad \,\, \lambda_p(\bfD^{-1/2}\bfSigma \bfD^{-1/2}) \le 1.
\]
Accordingly, the absolute values on the right-hand side of \eqref{e:simp:1} can be dropped. This motivates a simple estimation procedure for $\mathcal{S}(\bfSigma)$: First, create an estimator for $\bfD^{-1/2}\bfSigma \bfD^{-1/2}$. Second, use it to construct estimates of $\lambda_1(\bfD^{-1/2}\bfSigma \bfD^{-1/2}), \lambda_p(\bfD^{-1/2}\bfSigma \bfD^{-1/2})$. Third and finally, use the eigenvalue estimates for a plug-in estimate of \eqref{e:simp:1}. The first step is simple enough, since we already have the covariance estimator $\widehat{\bfSigma}_n$ and simply need to approximate $\mathbf{D}$. This can be done by extracting the diagonal blocks $\widehat{\bfSigma}_\mathbf{ii} \in \mathbb{R}^{\ell_i \times \ell_i}$ from $\widehat{\bfSigma}_n$ and defining $\hat \sigma_{i}^2 := \frac{1}{\ell_i}\tr(\widehat{\bfSigma}_\mathbf{ii})$. We can then set 
\begin{align*}
    \widehat{\mathbf{D}}:= \operatorname{diag}\bigl(
    \hat \sigma_1^2 \bfI_{\ell_1},
    \hat \sigma_2^2 \bfI_{\ell_2},
    \ldots,
    \hat \sigma_L^2 \bfI_{\ell_L}
    \bigr).
\end{align*}
Thus, we obtain the natural estimator for $\bfD^{-1/2}\bfSigma \bfD^{-1/2}$ by $\widehat{\mathbf{D}}^{-1/2}\widehat{\bfSigma}_n \widehat{\mathbf{D}}^{-1/2} $. The next step is now to approximate the extreme eigenvalues of $\bfD^{-1/2}\bfSigma \bfD^{-1/2}$. For fixed $p$, it is of course true that $\lambda_j(\bfD^{-1/2}\bfSigma \bfD^{-1/2}) = \lambda_j(\widehat{\mathbf{D}}^{-1/2}\widehat{\bfSigma}_n \widehat{\mathbf{D}}^{-1/2})+\mathcal{O}_\PR(n^{-1/2})$. However, in the proportional regime where $p \sim n$, this consistency is generally not true. Rather, there are several scenarios that depend on the limiting behavior of the spectrum, relative to the extreme eigenvalues. In this paper, we are concerned with what is sometimes called the \textit{supercritical regime}, where the matrix $\bfD^{-1/2}\bfSigma\bfD^{-1/2}$ has spiked extreme eigenvalues that are (in some sense) detached from the rest of the spectrum. In our numerical experiments, we will see some examples with supercritical eigenvalues, including the case of a panel with weak cross-sectional equicorrelations, which are a natural way to capture subtle deviations from a spherical model. A mathematical formulation of the supercritical regime is given in Assumption \ref{ass4_spikes} below. In this regime, the extreme eigenvalues of the empirical covariance are biased, but it is possible to give a bias correction, following \cite{zhang:zheng:pan:zhong:2022}. We therefore define
\begin{align*}
    \widehat{\bfSigma}_n^c &=   \widehat{\bfD}^{-1/2}\  \widehat{\bfSigma}_n  \widehat{\bfD}^{-1/2},
\end{align*}
and set
\begin{align}
    \hat\alpha_{1,n}^c & = -\lb-\frac{1-p/n}{\lambda_{1}(\hat\bfSigma_n^c)}+\frac{1}{n}\sum_{i\ne 1}\frac{1}{\lambda_{i}(\hat\bfSigma_n^c)-\lambda_{1}(\hat\bfSigma_n^c)}\rb^{-1},\label{eq_def_hat_alpha1}\\
    \hat\alpha_{p,n}^c & =-\lb-\frac{1-p/n}{\lambda_{p}(\hat\bfSigma_n^c)}+\frac{1}{n}\sum_{i\ne p}\frac{1}{\lambda_{i}(\hat\bfSigma_n^c)-\lambda_{p}(\hat\bfSigma_n^c)}\rb^{-1} . \label{eq_def_hat_alphap} 
\end{align}
These estimators are motivated by an identity linking the population eigenvalue $\lambda_k(\bfD^{-1/2}\bfSigma\bfD^{-1/2})$ to the deterministic location of the corresponding sample outlier through the companion Stieltjes transform of the bulk; see \eqref{eq_alpha_formula} for more details. Replacing the deterministic location by the observed sample eigenvalue $\lambda_k(\hat\bfSigma_n^c)$
and the Stieltjes transform by an empirical leave-one-out counterpart yields $\hat\alpha_{k,n}^c$. Related estimators appeared previously in \cite{zhang:zheng:pan:zhong:2022, lam:doernemann:dette:2025, bai:ding:2012}.\\
It is important to notice that the bias-corrected estimator of the smallest eigenvalue $\hat\alpha_{p,n}^c$ is only meaningfully defined if $p/n <1$; otherwise the smallest empirical eigenvalue is simply $0$.
In fact, our subsequent asymptotics for the smallest eigenvalue require the ratio $p/n$ to remain bounded away from $1$, which we ensure in our procedure by distinguishing the cases  $p/n \le 1-\theta$ and   $p/n>1-\theta$ for some positive constant $\theta$. In the case where  $p/n>1-\theta$, our method focuses only on the largest eigenvalue, which is formally attained by defining $\hat\alpha_{p,n}^c =\infty$ if  $p/n>1-\theta$ in the following. With this in mind, we continue presenting our discussion.\\
With the bias-corrected eigenvalue estimators in hand, we can now define the empirical version of the goodness-of-fit measure $\mathcal{S}(\boldsymbol{\Sigma})$ (following \eqref{e:simp:1}) as $\max\lb   \hat\alpha_{1,n}^c - 1    ,
       1 - \hat\alpha_{p,n}^c
    \rb $. Then, to test the hypothesis of approximate sphericity, we suggest the rescaled test statistic
    \begin{align} \label{eq_def_test_stat}
    \tilde T_n =\sqrt{n}  \max\lb   \hat\alpha_{1,n}^c - 1 - \Delta   ,
       1 - \Delta - \hat\alpha_{p,n}^c
    \rb.
\end{align}
For $c<1$, the analysis of $\tilde T_n$ depends on the asymptotic normality of the two-dimensional vector
\begin{align} \label{e:conv:joint}
\sqrt{n}
\begin{pmatrix}
    \hat\alpha_{1,n}^c - \lambda_1(\bfD^{-1/2}\bfSigma\bfD^{-1/2})\\
    \lambda_p(\bfD^{-1/2}\bfSigma\bfD^{-1/2})-\hat\alpha_{p,n}^c
\end{pmatrix}
\cond
Z=(Z_1,Z_2)^\top\sim\mathcal{N}_2(\mathbf{0},\mathbf{J}\mathbf{C}\mathbf{J}),
\end{align}
where $\mathbf{J}=\operatorname{diag}(1,-1)$ and $\mathbf{C}$ is the covariance matrix in Theorem \ref{thm_rmt}.

\begin{algorithm}[h]
\caption{Test for $\Delta$-Approximate Sphericity}
\label{alg:sphericity}
\begin{algorithmic}[1]
\Require Significance level $\alpha \in (0,1)$, threshold $\Delta > 0$, data $\mathbf{y}_1,\ldots,\mathbf{y}_n$, symmetric positive semidefinite covariance estimate $\widehat{\mathbf{C}}_n=(\widehat C_{n,ij})_{i,j=1}^2$ for the eigenvalue errors in Theorem \ref{thm_rmt}, number of Monte Carlo runs $B\in\mathbb N$, threshold $\theta \in (0,1)$
\Ensure Test decision
\State Compute $\hat{\alpha}^c_{1,n}$ (defined in \eqref{eq_def_hat_alpha1})
\If{$p/n \le 1-\theta$}
\State Compute $\hat{\alpha}^c_{p,n}$ (defined in \eqref{eq_def_hat_alphap})
\State Compute the test statistic $\tilde T_n$ (defined in \eqref{eq_def_test_stat})
\State Compute
       $A_1 := \hat{\alpha}^c_{1,n} - 1$ \quad and \quad
       $A_2 := 1 - \hat{\alpha}_{p,n}^c$
\State Set $i := \min\arg\max_{j \in \{1,2\}} A_j$ and define $I := \{i\}$
\If{$|A_1 - A_2| \le \log(n)/\sqrt{n}$}
    \State Update $I := \{1,2\}$
\EndIf
\Else
    \State Set $I := \{1\}$ and $\tilde T_n = \sqrt{n}\lb\hat\alpha_{1,n}^c - 1 - \Delta\rb$
\EndIf
\State Form $\widehat{\mathbf{C}}^{\mathrm{test}}_{n,I}:=((\mathbf{J}\widehat{\mathbf{C}}_n\mathbf{J})_{ij})_{i,j\in I}$
\State Conditionally on the data, generate independent vectors
$(Z_i^{(b)})_{i\in I}\sim\mathcal N_{|I|}(\mathbf0,\widehat{\mathbf{C}}^{\mathrm{test}}_{n,I})$, $b=1,\ldots,B$
\State Set $M_b:=\max_{i\in I}Z_i^{(b)}$ and order the values as $M_{(1)}\le\cdots\le M_{(B)}$
\State Set $q_{1-\alpha}^{B}:=M_{(\lceil B(1-\alpha)\rceil)}$
\If{$\tilde T_n > q_{1-\alpha}^B$}
    \State \Return \texttt{Reject approximate sphericity}
\Else
    \State \Return \texttt{Do not reject $\Delta$-approximate sphericity}
\EndIf
\end{algorithmic}
\end{algorithm}

The explicit form of the limiting covariance matrix $\mathbf{C}$ is given in Appendix \ref{sec_C}. For $c<1$, the convergence \eqref{e:conv:joint} implies the following asymptotic behavior of $\tilde T_n$, where different cases need to be distinguished. If $\mathcal{S}(\bfSigma) < \Delta$, the "interior of the null-hypothesis", it simply follows from \eqref{e:conv:joint} that $   \tilde T_n \overset{P}{\to} -\infty$. In the case where  $\mathcal{S}(\bfSigma) = \Delta$ we are still under the null hypothesis, but at the boundary and $ \tilde T_n $ converges to the following non-degenerate distribution
\begin{align} \label{e:boots:lim}
 \tilde T_n  \cond \max_{i\in I^\star} Z_i,
\end{align}
where $I^\star$ is defined below. Consequently, $ \tilde T_n $ asymptotically behaves like one of two normal distributions or like their maximum, depending on which extreme eigenvalue determines the spectral norm. Finally, when $\mathcal{S}(\bfSigma) > \Delta$, it follows that $   \tilde T_n \overset{P}{\to} \infty$. For $c\ge1$, the same conclusions hold with $I^\star=\{1\}$ and $Z_1\sim\mathcal N(0,C_{11})$ under Assumption \ref{ass4_spikes}(b). Jointly, these asymptotics suggest rejecting approximate sphericity whenever $\tilde T_n$ is large compared to the asymptotic quantiles of the right-hand side of \eqref{e:boots:lim}. A practically implementable test decision involves estimation of the covariance matrix $\mathbf{C}$ in Theorem \ref{thm_rmt}, recovery of the index set in \eqref{e:boots:lim} and then simulating quantiles of the limiting distribution using a Monte Carlo method, and a summary is given in Algorithm \ref{alg:sphericity}. In the next section, we turn to the formal justification of our approach.

\subsection{Theoretical guarantees}

For our mathematical analysis, some key definitions are needed.
First, let $c\in (0,\infty)$ be a constant of proportionality such that $p/n \to c$ as $n\to\infty$. This corresponds to our high-dimensional setting, where sample size and cross-sectional dimension are proportional to each other. 
Second, using the shorthand $\alpha_i := \lambda_i(\bfD^{-1/2}\bfSigma \bfD^{-1/2})$, we define the corresponding empirical spectral measure as follows:
\begin{align*}
    H_n = \frac{1}{p}\sum_{j=1}^p \delta_{\alpha_j}.
\end{align*}
$H_n$ summarizes the behavior of the bulk of eigenvalues, which influences the large-sample behavior of the extremal empirical eigenvalues.
To analytically describe the large-sample behavior, we will impose the standard assumption of $H_n$  
converging weakly to a limit called $H$ as $n\to\infty$. Third, we define for this limiting measure the function
\begin{align}\label{eq:def_phi}
    \phi(\lambda)=\lambda+c\int\frac{\lambda t}{\lambda-t}dH(t), \quad \lambda \notin \operatorname{supp}(H).
\end{align}
It turns out that the derivative of $\phi$ is important for characterizing supercritical eigenvalues. Finally, we define the set 
$I^\star$ as containing the active coordinate indices
\[
I^\star =
\begin{cases}
\{i\in\{1,2\}:a_i=\mathcal S(\bfSigma)\}, & c<1,\\
\{1\}, & c\ge1,
\end{cases}
\qquad
 a_1:=\alpha_1-1,\quad a_2:=1-\alpha_p,
\]
which identifies either the largest, the smallest or both eigenvalues (depending on whether the spikes are of equal or unequal sizes).
\begin{assumption}
\textnormal{
\begin{enumerate}[label=(A-\arabic*)]
\item \label{ass1_dim} (On the dimension) $p=p_n$ is a function of $n$ such that $p/n \to c \in (0,\infty)$ as $n\to\infty.$
\item \label{ass2_data}(On the data) Let $\bfGamma$ be a $p\times p$ matrix such that  $\bfGamma=\bfSigma^{1/2}$ and $\mathbf{Q}$ a $p\times n$ random matrix with i.i.d. entries $   q_{ij}$ satisfying 
    $\E(q_{11})=0$, $\E(q_{11}^2)=1$, and $\E(q_{11}^4)<\infty$.
The data matrix can be written as
$(\bfy_1-\mu, \ldots, \bfy_n-\mu) = \bfGamma \mathbf{Q} . $ 
\item \label{ass3_bulk}(On the bulk eigenvalues) The empirical spectral distribution  $(1/p)\sum_{i=1}^p \delta_{\alpha_i}$ converges weakly to a probability measure $H$.
\item  \label{ass4_spikes} (On the extremal eigenvalues) 
\begin{itemize}
    \item[(a)] If $c<1$, assume that $\alpha_1=\alpha_{\max}$ and $\alpha_p=\alpha_{\min}$ are fixed numbers in $(0,\infty)$, independent of $n$, and both have multiplicity one and are uniformly separated from the remaining population eigenvalues. Assume further that $\alpha_{\max},\alpha_{\min}\notin\operatorname{supp}(H)$, and that $\phi'(\alpha_{\max})>0$ and $\phi'(\alpha_{\min})>0$.
    \item[(b)] If $c\ge 1$, assume that $\alpha_1=\alpha_{\max}$ is a fixed number in $(0,\infty)$, independent of $n$, and has multiplicity one and is uniformly separated from the remaining population eigenvalues. Assume further that $\alpha_{\max}\notin\operatorname{supp}(H)$, that $\phi'(\alpha_{\max})>0$, and that $\mathcal{S}(\bfSigma) = \alpha_1-1$.
\end{itemize}
\item \label{ass5_groups} (On the group sizes) The number of groups $L \in \mathbb{N}$ does not depend on $n$. The group sizes $\ell_1,...,\ell_L$ depend on $n$ and satisfy $\ell_i/p \to c^{(i)}$ for $i=1,...,L$, where $ c^{(i)}>0$ and $\sum_{i=1}^L  c^{(i)}=1$.
\end{enumerate}
}
\end{assumption}

\begin{remark}
Assumption \ref{ass4_spikes}(a) is needed to establish
the joint central limit theorem for both corrected
extreme eigenvalue estimators. The algorithm itself,
however, uses sample eigenvalues and does not require
both population extreme eigenvalues to be simple.
In particular, a repeated population eigenvalue in
the bulk does not by itself make the corresponding
empirical estimator undefined.

The testing conclusions remain valid if only the
active extreme eigenvalue is simple and supercritical,
provided that the other coordinate is asymptotically
excluded from the selected set $I$.
For example, if $\mathcal S(\bfSigma)=\alpha_{\max}-1$,
a sufficient additional condition is that, for some
$\varepsilon>0$,
\[
\mathbb P\left(
1-\hat\alpha_{p,n}^c
\le \alpha_{\max}-1-\varepsilon
\right)\longrightarrow1.
\]
Together with consistency of $\hat\alpha_{1,n}^c$,
this condition ensures that $I=\{1\}$ with probability
tending to one. The conclusions then follow from the
univariate central limit theorem and consistency of
the variance estimator for the upper coordinate,
without requiring the smallest population eigenvalue
to be simple or supercritical.
\end{remark}

\noindent Assumption \ref{ass1_dim} implies a high-dimensional setting, with comparable sample size and dimension. \ref{ass2_data} is a standard assumption in random matrix theory, restricting the structure of higher-order dependence; it is satisfied, in particular, by multivariate normal vectors. \ref{ass3_bulk} is also standard, and implies that the bulk distribution stabilizes in our asymptotic theory (which in turn guarantees that the distribution of leading eigenvalues stabilizes). \ref{ass4_spikes} requires that the extreme eigenvalues used in the test are simple and uniformly separated from the remaining population eigenvalues, which is key for the asymptotic normality of empirical eigenvalues. Another prerequisite for that is supercriticality (large enough spikes), which is expressed by the condition on a positive derivative of $\phi$. The distinction between $c<1$ and $c \ge 1$ reflects whether the smallest empirical eigenvalue is used ($c<1$) or excluded from the procedure ($c\ge1$). In the case where $c\ge 1$, we have added the condition $\mathcal{S}(\bfSigma) = \alpha_1-1$ to make the statements of our results easier. In fact it can be dropped without affecting asymptotic level control under the null hypothesis, but then equality to $\alpha$ at every boundary point is no longer guaranteed. Under the alternative, it matters, however, when stating consistency, because a test based only on the largest eigenvalue can, of course, only detect violations of $H_0$ in the largest eigenvalue, and not in the smallest one. The last condition \ref{ass5_groups} has been verbally stated earlier and assumes that the number of groups in our asymptotic theory is fixed and all groups have comparable sizes. \\
Before turning to our main theoretical results, we present a minimal example illustrating the deviation measure.

\begin{example}
\textnormal{
As a minimal illustration, we consider an equicorrelation example. To make the presentation easier, suppose that
there is only one group, \(L=1\), and that \(\sigma^2=1\), so that
\(\mathbf D=\mathbf I_p\). Then, let
\[
    \boldsymbol{\Sigma}
    =
    \begin{pmatrix}
        1      & \rho  & \cdots & \rho \\
        \rho   & 1     & \cdots & \rho \\
        \vdots & \vdots & \ddots & \vdots \\
        \rho   & \rho  & \cdots & 1
    \end{pmatrix}
    \in \mathbb R^{p\times p}, \qquad \textnormal{where} \qquad   \rho=\rho_p:=\frac{\gamma}{p-1},
\]
for some fixed \(\gamma\ge0\). For all sufficiently large \(p\), this is a
positive definite covariance matrix, with eigenvalues 
\[
    \alpha_1=1+(p-1)\rho=1+\gamma,
    \qquad
    \alpha_2=\cdots=\alpha_p=1-\rho
    =
    1-\frac{\gamma}{p-1}.
\]
As a consequence, the empirical spectral measure \(H_n\) converges weakly to its limit
\(H=\delta_1\), and the map $\phi$ is given by
\[
    \phi(\lambda)
    =
    \lambda\left(1+\frac{c}{\lambda-1}\right),
    \qquad
    \phi'(\lambda)
    =
    1-\frac{c}{(\lambda-1)^2}.
\]
For $\gamma>0$, the upper spike is \(\alpha_1=1+\gamma\), and the supercriticality condition can be expressed through the aspect ratio $c$ as follows:
\[
    \phi'(\alpha_1)>0
    \qquad\Longleftrightarrow\qquad
    \gamma>\sqrt c.
\]
For $\gamma>0$ and $p>2$, only the largest eigenvalue is active, i.e. $I^\star = \{1\}$.  \\
Equicorrelation is a stylized model for weak but pervasive cross-sectional
dependence. Although each individual correlation satisfies
\(\rho_p\to0\), the aggregate spectral deviation from sphericity is non-vanishing for $\gamma>0$:
\[
    \mathcal S(\boldsymbol{\Sigma})
    =
    \|\boldsymbol{\Sigma}-\mathbf I_p\|
    =
    \gamma.
\]
In this particular setup, the relevant hypothesis becomes
\[
    H_0:\gamma\le \Delta
    \qquad\text{versus}\qquad
    H_1:\gamma>\Delta.
\]
The parameter \(\gamma\) (and thus the threshold $\Delta$) can be interpreted in the context of cross-sectional pooling. Suppose, for instance,
that all components have the same mean
\(\mathbb E[\mathbf{y}_{j,i}]=\mu\), and that the aim is to estimate \(\mu\). We consider the pooled average
\[
    \overline{\bar{\mathbf{y}}}
    :=
    \frac{1}{np}\sum_{j=1}^n\sum_{i=1}^p \mathbf{y}_{j,i},
\]
which has mean squared error (MSE)
\[
    Q_\gamma
    :=
    \mathbb E\bigl[(\overline{\bar{\mathbf{y}}}-\mu)^2\bigr]
    =
    \frac{1+\gamma}{np}.
\]
In the independent case \(\gamma=0\), we get the smallest possible value $  Q_0= 1/(np)$. To assess efficiency relative to the independent case, it makes sense to consider the relative increase in MSE, which is simply
\[
    \frac{Q_\gamma-Q_0}{Q_0}=\gamma.
\]
Consequently, the tolerance parameter \(\Delta\) can be interpreted as the maximum tolerable relative increase in MSE (thus $\Delta$ quantifies the efficiency loss due to cross-sectional dependence).
Similar interpretations
apply to grouped means, and to the cross-sectionally pooled slope estimators in
the regression setting considered below.
}
\end{example}

\begin{theorem}
\label{thm_main_test_sph}
Suppose that Assumptions \ref{ass1_dim}--\ref{ass5_groups} are
satisfied and that the covariance limits required in Theorem \ref{thm_rmt} exist. Let $\alpha\in(0,1)$ and $\Delta>0$ be fixed. Moreover, let
$\theta\in(0,1)$ and, if $c<1$, assume additionally that
$\theta<1-c$. Suppose that the covariance estimate
$\widehat{\mathbf C}_n$ supplied to Algorithm \ref{alg:sphericity}
is data-measurable, symmetric and positive semidefinite, and that
\[
\begin{cases}
\|\widehat{\mathbf C}_n-\mathbf C\|\xrightarrow{\mathbb P}0,
\quad \mathbf C\text{ is positive definite}, & c<1,\\[1mm]
\widehat C_{n,11}\xrightarrow{\mathbb P}C_{11}>0, & c\ge1,
\end{cases}
\]
where $\mathbf C$ and $C_{11}$ are the limiting covariances in
Theorem \ref{thm_rmt}. Let $B=B_n\to\infty$
and denote the output of Algorithm \ref{alg:sphericity} by $A$. Then
\begin{align*}
\lim_{n\to\infty}
\mathbb{P}\bigl(
A=\textnormal{\texttt{``Reject approximate sphericity''}}
\bigr)
=
\begin{cases}
0,      & \mathcal{S}(\bfSigma)<\Delta,\\
\alpha, & \mathcal{S}(\bfSigma)=\Delta,\\
1,      & \mathcal{S}(\bfSigma)>\Delta.
\end{cases}
\end{align*}
\end{theorem}

Although not covered explicitly by our assumptions, the same argument applies when $c<1$ and the procedure uses only one supercritical extreme eigenvalue, provided that this eigenvalue determines $\mathcal{S}(\bfSigma)$. In this case, the other coordinate is omitted from the statistic and the calibration.\\
The proof of the main theorem relies on the following central limit theorem for the corrected extreme eigenvalue estimators. We present it here, defer its proof to Section \ref{sec_proof_thm_rmt}, and give the definition of the covariance matrix $\mathbf{C}$ in Appendix \ref{sec_C}.
\begin{theorem} \label{thm_rmt}
Suppose that Assumptions \ref{ass1_dim}--\ref{ass5_groups} are satisfied.
Assume that the limits defining $\Omega_{11}$, $\Omega_{22}$ and $\Omega_{12}$ in \eqref{eq:def_Omega} exist when $c<1$, and that the limit defining $\Omega_{11}$ exists when $c\ge1$. 
\begin{enumerate}[label=(\alph*)]
    \item \label{thm_part_a}
    If $c<1$, then
    \begin{align*}
        \sqrt{n} \lb \hat\alpha_{1,n}^c - \alpha_{\max}, \hat\alpha_{p,n}^c - \alpha_{\min} \rb \cond \mathcal{N}_2 (\mathbf{0}, \mathbf{C}), \quad n\to\infty.
    \end{align*}
    \item \label{thm_part_b} If $c \geq 1$, then
    \begin{align*}
        \sqrt{n} \lb \hat\alpha_{1,n}^c - \alpha_{\max}\rb \cond \mathcal{N} (0, C_{11}), \quad n\to\infty.
    \end{align*}
\end{enumerate}
    
\end{theorem}

The Gaussian approximation above can also be used to construct asymptotically valid confidence intervals for $\mathcal{S}(\bfSigma)$, using lower and upper quantiles of the limiting distribution simulated in Algorithm \ref{alg:sphericity}.

\subsection{Extension to linear regression models}

In the previous section, we have focused on the standard high-dimensional model, where vectors (with a certain hypothesized approximate group structure) are tested for an approximately groupwise spherical covariance matrix. In this section, we discuss a statistical application to regression panels, which play an important role in econometrics and beyond. Consider the following panel regression model
\[
y_{i,t}= \gamma_i+\bfx_{i,t}'\bfbeta+\varepsilon_{i,t}, \quad i=1,...,p, \,\, t=1,...,n
\]
where $\bfbeta \in \R^k$ is the slope vector (identical across all individuals), $i$ the cross-sectional component, $\gamma_i$ the individual-specific intercept and $t$ the time component. The cross-section consists of $p$ individuals and there are a total of $n$ time points. Assuming (groupwise) sphericity for the model errors $\varepsilon_{i,t}$ is common in panel applications, and the validity of subsequent inference procedures frequently hinges on it. Nevertheless, it is typically clear that the assumption does not hold exactly in practice. This makes panel models a natural setting for testing approximate sphericity. 
To recover the residuals in the panel model, we first use an empirical centering in each component
\[
\tilde y_{i,t}= (\bfx_{i,t} - \bar \bfx_i)'\bfbeta+\varepsilon_{i,t}-\bar \varepsilon_i, \quad i=1,...,p, \,\, t=1,...,n,
\]
and stacking all the individuals at time $t$ gives us the vectorized version
\[
\tilde \bfy_{t} = (\bfX_{t} - \bar \bfX)'\bfbeta+(\boldsymbol\varepsilon_{t}-\bar {\boldsymbol\varepsilon}).
\]
More precisely, this gives us $p$-dimensional vectors $\tilde \bfy_{t}$ and $\boldsymbol\varepsilon_{t}-\bar {\boldsymbol\varepsilon}$, and $k\times p$-dimensional matrices $\bfX_{t} - \bar \bfX$.

\begin{assumption}
\textnormal{
\begin{enumerate}[label=(AP-\arabic*)]
\item \label{ass1_panel_dim} $p=p_n$ increases with $n$ such that $p/n \to c \in (0,\infty)$ as $n\to\infty$, while $k$ remains fixed.
\item \label{ass2_panel_data}
The regressor matrices admit the decomposition
\[
    \bfX_t^\top=\mathbf M_n+\bfG^{1/2}\mathbf Z_t,
    \qquad t=1,\ldots,n,
\]
where $\mathbf M_n\in\mathbb R^{p\times k}$ is deterministic
and $\bfG$ is a deterministic positive semidefinite
$p\times p$ matrix satisfying $\sup_n\|\bfG\|<\infty$.
The entries of the $p\times k$ matrices $\mathbf Z_t$
are i.i.d., with mean zero,
unit variance, and uniformly bounded finite
fourth moment.
\item \label{ass3_estimator_beta} Vectors $\boldsymbol{\varepsilon}_t$ are i.i.d., with covariance matrix $\bfSigma^{\varepsilon}$ satisfying $\sup_n\|\bfSigma^{\varepsilon}\|<\infty$, and admit the decomposition $\boldsymbol{\varepsilon}_t-\E(\boldsymbol{\varepsilon}_1)=(\bfSigma^{\varepsilon})^{1/2}\mathbf{v}_t$. The elements of $\mathbf{v}_t$ are i.i.d. across entries and time points, centered, with unit variance and uniformly bounded fourth moments.
\item \label{ass3_panel_noise} The estimator $\hat{\bfbeta}$ satisfies $\|\bfbeta-\hat\bfbeta\|=o_\PR(n
^{-1/2})$. 
\end{enumerate}
}
\end{assumption}
Assumption \ref{ass1_panel_dim} states that we are in a panel model where the cross-sectional dimension is non-negligible compared to the number of time points. The constant of proportionality may even be much larger than $1$, describing panels with large cross-section. This is particularly relevant in econometric applications, where the number of individuals often exceeds the number of time observations \citep{hansen:2007, pesaran:yamagata:2008}. The separability and moment assumptions \ref{ass2_panel_data} and \ref{ass3_estimator_beta} are standard in random matrix theory. Finally, Assumption \ref{ass3_panel_noise} is a consistency result for a slope estimator $\hat \bfbeta$ that pools all observed data. It is well-known that under certain standard assumptions, including exogeneity and (groupwise-)sphericity, one can prove the much stronger result $\|\bfbeta-\hat\bfbeta\|=\mathcal{O}_\PR((np)
^{-1/2})=\mathcal{O}_\PR(n
^{-1})$. See \citet[Theorem~3]{hansen:2007} and \citet[equation~(3)]{stock:watson:2008}. 
We do not impose this stronger result, first because our theory does not require it, and second (and more importantly), because it is easy to see that under stronger cross-sectional dependence, a rate of $\mathcal{O}_\PR(n
^{-1})$ may become unattainable, while $o_\PR(n
^{-1/2})$ may still hold. Since proving specific consistency results with rates is not our focus here, we have only imposed the high-level assumption \ref{ass3_panel_noise}.\\
If the true value of $\bfbeta$ were known, we could simply retrieve the centered errors $\boldsymbol\varepsilon_t-\bar{\boldsymbol\varepsilon}$, compute the empirical covariance of the model errors
\[
\hat{\bfSigma}^{\varepsilon}=\frac{1}{n}\sum_{t=1}^n \big[\boldsymbol\varepsilon_{t}-\bar{\boldsymbol\varepsilon}\big]\big[\boldsymbol\varepsilon_{t}-\bar{\boldsymbol\varepsilon}\big]'
\]
and apply the theory from the previous section. Since only $\hat \bfbeta$ is known, we only have the empirical covariance of the empirical residuals at our disposal, i.e.
\begin{align}\label{eq:est_cov}
\hat \bfSigma = \frac{1}{n}\sum_{t=1}^n \big[\tilde \bfy_{t}-(\bfX_{t} - \bar \bfX)'\hat \bfbeta\big]\big[ \tilde \bfy_{t} -(\bfX_{t} - \bar \bfX)'\hat \bfbeta \big]'.  
\end{align} 
Yet, if the estimate $\hat \bfbeta$ for $\bfbeta$ is good enough, we have a fairly clear proof strategy. We compare the empirical covariance of the residuals, $\hat \bfSigma$, with the empirical covariance of the true errors, $\hat{\bfSigma}^{\varepsilon}$, and show that their distance in operator norm is $o_\PR(n^{-1/2})$. This also gives us the same approximation rate for their eigenvalues. Indeed, Weyl's inequality yields
\begin{align*}
    \max_{1\le j\le p}
    \big|\lambda_j(\hat{\bfSigma})-\lambda_j(\hat{\bfSigma}^{\varepsilon})\big|
    \leq \|\hat{\bfSigma}-\hat{\bfSigma}^{\varepsilon}\|
    =o_\PR(n^{-1/2}).
\end{align*}
Thus, the effect of estimating the regression slope on the empirical eigenvalues is negligible at the $n^{-1/2}$ scale used in our asymptotic results. This provides the key approximation for transferring the theory from the previous section to the empirical residuals. The following proposition establishes the required operator norm bound.
\begin{proposition} Under Assumptions \ref{ass1_panel_dim}--\ref{ass3_panel_noise}, it follows that
    \begin{align*}
        \|\hat \bfSigma - \hat{\bfSigma}^{\varepsilon}\|=o_\PR(n^{-1/2})
    \end{align*}
\end{proposition}
The proof is relatively short, so we provide it here.
\begin{proof}
    We start by decomposing the r.h.s. of \eqref{eq:est_cov} as $\hat \bfSigma =  \hat{\bfSigma}^{\varepsilon} + \hat \bfSigma_2 +\hat \bfSigma_3+\hat \bfSigma_4$, where
\begin{align*}
    \hat \bfSigma_2 := &\frac{1}{n}\sum_{t=1}^n \big[ (\bfX_{t} - \bar \bfX)'(\bfbeta-\hat \bfbeta)\big]\big[\boldsymbol\varepsilon_{t}-\bar {\boldsymbol\varepsilon}\big]'\\
    \hat \bfSigma_3 := &\frac{1}{n}\sum_{t=1}^n \big[\boldsymbol\varepsilon_{t}-\bar {\boldsymbol\varepsilon} \big]\big[ (\bfX_{t} - \bar \bfX)'(\bfbeta-\hat \bfbeta)\big]'\\
    \hat \bfSigma_4 := &\frac{1}{n}\sum_{t=1}^n \big[ (\bfX_{t} - \bar \bfX)'(\bfbeta-\hat \bfbeta)\big]\big[ (\bfX_{t} - \bar \bfX)'(\bfbeta-\hat \bfbeta)\big]'.
\end{align*}
Since $\|\hat\bfSigma-\hat\bfSigma^{\varepsilon}\|\leq\sum_{i=2}^4\|\hat \bfSigma_i\|$, it is sufficient to show that $\|\hat \bfSigma_i\|=o_\PR(1/\sqrt{n}), i=2,3,4$. 
We start with $i=2$:
\begin{align*}
    \hat \bfSigma_2 = &\frac{1}{n}\sum_{t=1}^n \big[ (\bfX_{t} - \bar \bfX)'(\bfbeta-\hat \bfbeta)\big]\big[\boldsymbol\varepsilon_{t}-\bar {\boldsymbol\varepsilon}\big]'
\end{align*}
We rewrite $\hat{\bfSigma}_2$ as the product of three matrices:
\begin{align}
    \hat\bfSigma_2=\frac{1}{n}\mathcal{X}'(\bfI_n\otimes (\bfbeta-\hat\bfbeta))\mathcal{E}',
\end{align}
where 
\begin{align*}
    &\mathcal{X}=\begin{pmatrix}
\bfX_1-\bar \bfX\\
\ldots\\
\bfX_n-\bar \bfX
    \end{pmatrix} \in \R^{ (nk) \times p},
 \qquad \bfI_n\otimes (\bfbeta-\hat\bfbeta)=\begin{pmatrix}
        \bfbeta-\hat\bfbeta & \boldsymbol{0}_k&\ldots&\boldsymbol{0}_k\\
        \boldsymbol{0}_k&\ldots&\ldots&\boldsymbol{0}_k\\
        \boldsymbol{0}_k&\boldsymbol{0}_k&\ldots &\bfbeta-\hat\bfbeta
    \end{pmatrix} \in \R^{(nk)\times n}
   , \\
    &\mathcal{E}=\begin{pmatrix}
        \boldsymbol\varepsilon_{1}-\bar {\boldsymbol\varepsilon},&\ldots,&\boldsymbol\varepsilon_{n}-\bar {\boldsymbol\varepsilon}
    \end{pmatrix}
    \in \R^{p\times n}.
\end{align*}
Then due to submultiplicativity of the spectral norm
\begin{align} \label{eq_sigma2_decomp}
    \|\hat\bfSigma_2\|\leq\frac{1}{n}\|\mathcal{X}\|\|I_n\otimes (\bfbeta-\hat\bfbeta)\|\|\mathcal{E}\|.
\end{align}
We analyze each of the factors on the r.h.s. separately.  First,
\begin{align*}
    \|\bfI_n\otimes (\bfbeta-\hat\bfbeta)\|=\sqrt{\|(\bfI_n\otimes (\bfbeta-\hat\bfbeta))'(\bfI_n\otimes (\bfbeta-\hat\bfbeta))\|}=\sqrt{\|\|\bfbeta-\hat\bfbeta\|^2\bfI_n\|}=o_\PR(1/\sqrt{n}).
\end{align*}

Due to Assumptions~\ref{ass1_panel_dim} and~\ref{ass3_estimator_beta}, the matrix $\frac{1}{n}\mathcal{E}\mathcal{E}'$ has bounded spectral norm in probability as $n\rightarrow\infty$, i.e. $\|\mathcal{E}\mathcal{E}'\|=\mathcal{O}_\PR(n)$, see Appendix~\ref{sec:app_rmt}. From this follows
\begin{align*}
    \|\mathcal{E}\|=\mathcal{O}_\PR(\sqrt{n}).
\end{align*}

   By Assumption \ref{ass2_panel_data},
$\bfX_t'=\mathbf M_n+\bfG^{1/2}\mathbf Z_t$.
Hence, writing
$\bar{\mathbf Z}=n^{-1}\sum_{t=1}^n\mathbf Z_t$,
we obtain
\[
    (\bfX_t-\bar{\bfX})'
    =\bfG^{1/2}(\mathbf Z_t-\bar{\mathbf Z}).
\]
Consequently,
\begin{align*}
    \mathcal X'
    =\bfG^{1/2}
      (\mathbf Z_1-\bar{\mathbf Z},\ldots,
       \mathbf Z_n-\bar{\mathbf Z})=\bfG^{1/2}\mathcal Z\Phi,
\end{align*}
where
$\mathcal Z=(\mathbf Z_1,\ldots,\mathbf Z_n)
\in\mathbb R^{p\times(nk)}$ and
\[
    \Phi
    =\left(\bfI_n-\frac1n\mathbf1_n\mathbf1_n'\right)
      \otimes\bfI_k.
\]
Here $\mathbf1_n$ denotes the vector of $n$ ones.
The matrix $\Phi$ centers each regressor separately
across the $n$ time points and is a deterministic
orthogonal projection. Thus, using $\Phi\Phi'=\Phi$,
we obtain
\begin{align}
    \mathcal X'\mathcal X
    =\bfG^{1/2}\mathcal Z\Phi\mathcal Z'\bfG^{1/2}.
\end{align} 
Since $\|\Phi\|=1$, we have 
$\|\mathcal{X}'\mathcal{X}\| \leq \|\mathbf{G}\| \|\mathcal{Z} \|^2 =
\mathcal{O}_\PR(n)$.
Combining these bounds with \eqref{eq_sigma2_decomp}, we obtain
\begin{align}
\|\hat\bfSigma_2\|\le n^{-1}\mathcal{O}_\PR(n^{1/2}) o_\PR(n^{-1/2})  \mathcal{O}_\PR(n^{1/2}) =o_\PR(n^{-1/2})   .
\end{align}
Since $\hat{\bfSigma}_2 = \hat{\bfSigma}_3'$, the bound for $\|\hat{\bfSigma}_3\| = \| \hat{\bfSigma}_2 \|$ follows. 
In order to bound $\|\hat{\bfSigma}_4\|$, we use the notation from above and write
\begin{align}
    \hat\bfSigma_4=\frac{1}{n}\mathcal{X}'(\bfI_n\otimes (\bfbeta-\hat\bfbeta))(\bfI_n\otimes (\bfbeta-\hat\bfbeta))'\mathcal{X}.
\end{align}
Then we have immediately
\begin{align}
    \|\hat\bfSigma_4\|\leq\frac{1}{n}\|\mathcal{X}\|^2\|(\bfI_n\otimes (\bfbeta-\hat\bfbeta))\|^2=o_\PR(1/n).
\end{align}

\end{proof}

To apply our testing procedure to regression residuals,
suppose additionally that the errors satisfy
Assumptions \ref{ass1_dim}--\ref{ass5_groups} and that
their group-average variances are uniformly bounded
away from zero.
Under these conditions, the preceding approximation
remains valid after empirical groupwise normalization.
Weyl's inequality and the separated-eigenvalue argument
in Lemma \ref{lem:alpha_hat-alpha_hat_c} then imply that
the corrected extreme eigenvalue estimators computed
from residuals and from true errors differ by
$o_\PR(n^{-1/2})$.
The conclusion of Theorem \ref{thm_main_test_sph}
therefore extends to residuals under its remaining
conditions, provided that the covariance estimator
used for calibration satisfies its consistency
assumption.

\section{Simulation study}\label{sec:exp}
We investigate the finite-sample properties of our new procedure. For this reason, we consider empirical rejection probabilities across four different
scenarios. Throughout, we use $p=n/2$, with
$n\in\{100,200,500,1000\}$, and generate $1{,}000$ independent Monte Carlo
replications for each setting. We fix the relevant threshold at $\Delta=2$.
In all scenarios, we construct covariance matrices such that $\mathcal{S}(\Sigma)=b$, with $b$ a parameter that we vary between $1.6$ and $3.4$ to generate data under the null hypothesis and under the alternative. In our earlier terminology,
 $b<2$ corresponds to the interior of the null hypothesis (we expect rejection rates lower than $\alpha$), $b=2$
to its boundary (we expect rejection rates approximating $\alpha$), and $b>2$ to the alternative (we expect high rejection rates corresponding to power). We now describe the models under consideration. Models MI--MIII contain
one group, while MIV contains two groups.
\begin{itemize}
\item[]\textit{MI:}
We generate i.i.d. Gaussian observations
$\mathbf{y}_i\sim\mathcal{N}_p(0,\Sigma_{\mathrm{I}}(b))$.
The covariance matrix is diagonal, with first entry $1+b$ and all
remaining entries equal to $1-b/(p-1)$.  Consequently,
$D=I_p$ and it is easy to calculate that $\mathcal{S}(\Sigma_{\mathrm{I}}(b))=b$ holds exactly.
\item[]\textit{MII:}
Let $F_{\mathrm{B}}$ denote the distribution function of the
$\operatorname{Beta}(2,2)$ distribution and set
$q_{j,p}=F_{\mathrm{B}}^{-1}((j-3/2)/(p-1))$ for $j=2,\ldots,p$.
We define a diagonal matrix $\Lambda_b$ with first entry $(1+b)/2$
and remaining entries $(1-b/(p-1))q_{j,p}$.
The covariance matrix is $\Sigma_{\mathrm{II}}(b)=Q\Lambda_bQ^\top$,
where $Q$ is drawn from the Haar distribution on the orthogonal group
once for each $p$ and held fixed across parameter values and
replications. Observations are generated as
$\mathbf{y}_i=Q\Lambda_b^{1/2}Q^\top\mathbf{z}_i$, where the entries
of $\mathbf{z}_i$ are independent copies of $\sqrt{3/5}\,\xi$ with
$\xi\sim t_5$. These entries have mean zero and variance one, but are heavy-tailed.
By symmetry, the average of the $q_{j,p}$ is exactly $1/2$,
so $D=I_p/2$. The largest normalized eigenvalue is $1+b$,
and we obtain $\mathcal{S}(\Sigma_{\mathrm{II}}(b))=b$.
\item[]\textit{MIII:}
We generate i.i.d. Gaussian observations with covariance
$\Sigma_{\mathrm{III}}(b)=(1-\rho_p)I_p+
\rho_p\mathbf{1}_p\mathbf{1}_p^\top$, where $\rho_p=b/(p-1)$ (this is the equicorrelation case).
All components have variance one and all pairwise correlations
equal $\rho_p$. Thus, $D=I_p$, and the eigenvalues are $1+b$
and $1-b/(p-1)$, the latter with multiplicity $p-1$.
It follows that $\mathcal{S}(\Sigma_{\mathrm{III}}(b))=b$.
\item[]\textit{MIV:} We consider a two-group scenario. We therefore
 divide the components into two groups of sizes
$\ell_1=\lfloor p/2\rfloor$ and $\ell_2=p-\ell_1$.
The observations are Gaussian with a diagonal covariance matrix.
In the first group, the first variance is $\frac{2}{3}(1+b)$
and the remaining $\ell_1-1$ variances are
$\frac{2}{3}(1-b/(\ell_1-1))$. All variances in the second group
equal $1/3$. The group averages are therefore exactly $2/3$ and
$1/3$, giving $D=\operatorname{diag}(\frac{2}{3}I_{\ell_1},
\frac{1}{3}I_{\ell_2})$.
After variance normalization, the eigenvalues are $1+b$,
$1-b/(\ell_1-1)$ with multiplicity $\ell_1-1$, and one with
multiplicity $\ell_2$. Thus, $\mathcal{S}(\Sigma_{\mathrm{IV}}(b))=b$.
\end{itemize}
Throughout this section, we focus on the test based on the largest eigenvalue only (which determines $\mathcal{S}(\Sigma)$ in all models).
For limiting variance estimation, we use $\hat C_{11}=2(\hat{\alpha}_{1,n}^{c})^{4}\hat{\tau}_{\max}$,
with $\hat{\alpha}_{1,n}^{c}$ defined in
\eqref{eq_def_hat_alpha1} and $\hat{\tau}_{\max}$ defined in
Section~\ref{sec_C}. \\
Table~\ref{tab:simulation_boundary} reports rejection rates at
$b=\Delta$ for nominal levels $5\%$ and $10\%$.
Across all Gaussian scenarios MI, MIII and MIV, our test is slightly conservative, particularly for the smaller sample sizes
$n=100$ and $n=200$; level approximation improves for larger $n$.
MII is a scenario with quite heavy-tailed data. As may be expected, the test somewhat exceeds both nominal levels, though the approximation remains reasonable.
\begin{table}[htbp]
\centering
\caption{Empirical rejection probabilities at the null boundary
$b=\Delta=2$, based on $1{,}000$ replications.}
\label{tab:simulation_boundary}
\begin{tabular}{rrccccc}
\toprule
$n$ & $p$ & $\alpha$ & MI & MII & MIII & MIV \\
\midrule
100  & 50  & 0.05 & 0.034 & 0.060 & 0.020 & 0.026 \\
100  & 50  & 0.10 & 0.072 & 0.134 & 0.065 & 0.062 \\
\midrule
200  & 100 & 0.05 & 0.028 & 0.060 & 0.024 & 0.025 \\
200  & 100 & 0.10 & 0.079 & 0.108 & 0.075 & 0.072 \\
\midrule
500  & 250 & 0.05 & 0.040 & 0.061 & 0.043 & 0.040 \\
500  & 250 & 0.10 & 0.089 & 0.123 & 0.092 & 0.089 \\
\midrule
1000 & 500 & 0.05 & 0.039 & 0.069 & 0.037 & 0.043 \\
1000 & 500 & 0.10 & 0.089 & 0.127 & 0.080 & 0.086 \\
\bottomrule
\end{tabular}
\end{table}
$ $\\
Next, we consider power properties. Figure~\ref{fig:simulation_power} shows rejection frequencies
over a grid of $b$-values at the nominal $5\%$ level, with steps of size $0.1$. The boundary of the null hypothesis is marked by a vertical gray line, and the nominal level by a horizontal gray line. Values to the right of the vertical gray line reflect power.
As may be expected, larger samples yield more power with steeper power curves. For $n=1000$, power is close to one by
$b=2.7$ in all four scenarios, while for $n=100$ it is around
$0.8$ at $b=3.4$. Rejection rates at the boundary correspond to those reported in Table~\ref{tab:simulation_boundary} above. Rejection rates to the left of the boundary ($b<2$) reflect the interior of the null hypothesis, where we expect, and indeed observe, rejection rates below $\alpha=5\%$. The rejection curves appear to cross approximately at the boundary $b=\Delta$, with larger samples leading to fewer rejections in the interior of the null hypothesis. This agrees with Theorem~\ref{thm_main_test_sph}, which predicts a limiting rejection probability of zero in the interior and $\alpha$ at the boundary.

\begin{figure}[htbp]
\centering
\includegraphics[width=\textwidth]{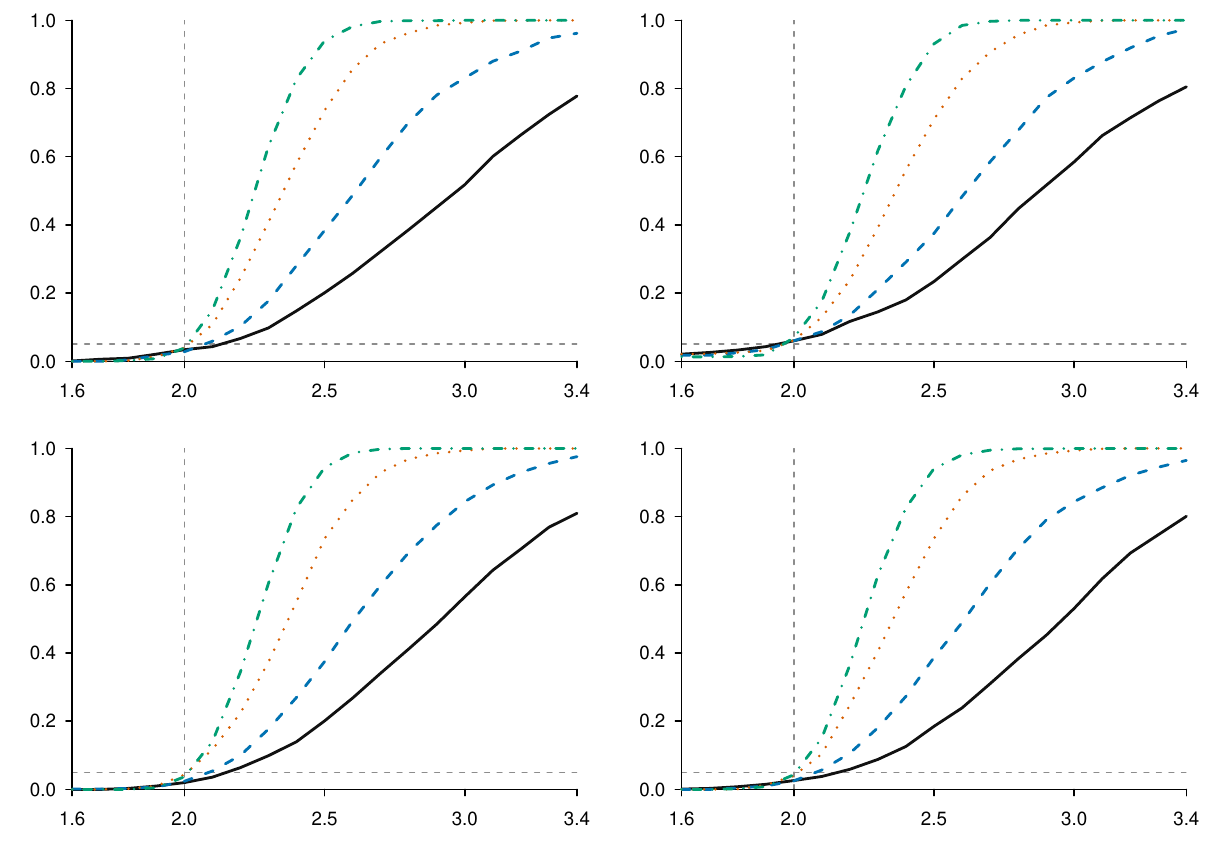}
\caption{Empirical rejection frequencies at nominal level $5\%$.
The panels show MI and MII in the top row and MIII and MIV in
the bottom row. The horizontal axis gives $b$ and the vertical
axis gives the rejection frequency. Black solid, blue dashed,
orange dotted and green dot-dashed curves correspond to
$n=100,200,500,1000$, respectively, with $p=n/2$.
The gray dashed lines mark $b=\Delta=2$ and the nominal
rejection probability $0.05$.}
\label{fig:simulation_power}
\end{figure}

\newpage
\section{Proofs}

\subsection{Proof of Theorem \ref{thm_main_test_sph}}

\begin{proof}[Proof of Theorem \ref{thm_main_test_sph}]
We focus only on the case of $c<1$; the case $c \ge 1$ is similar but simpler.\medskip

\noindent Since $p/n\to c<1$ and $\theta<1-c$, Algorithm
\ref{alg:sphericity} uses both extreme eigenvalue estimators
for all sufficiently large $n$. By Theorem \ref{thm_rmt},
\begin{align} \label{eq:joint_conv}
    \sqrt n
    \begin{pmatrix}
        \hat\alpha_{1,n}^c-\alpha_{\max}\\
        \alpha_{\min}-\hat\alpha_{p,n}^c
    \end{pmatrix}
    \xrightarrow{d}
    (Z_1,Z_2)^\top
    \sim\mathcal N_2(\mathbf0,\mathbf J\mathbf C\mathbf J),
    \qquad \mathbf J=\operatorname{diag}(1,-1).
\end{align}
Let $I_n$ denote the index set selected by the algorithm.
The population eigenvalues are fixed by
Assumption \ref{ass4_spikes}. If
$\alpha_{\max}-1>1-\alpha_{\min}$, consistency of the
estimators and $\log(n)/\sqrt n\to0$ imply
$\mathbb P(I_n=\{1\})\to1$. If
$\alpha_{\max}-1<1-\alpha_{\min}$, the same argument gives
$\mathbb P(I_n=\{2\})\to1$. If the two deviations coincide,
the joint convergence in \eqref{eq:joint_conv} gives
\[
    \sqrt n(\hat\alpha_{1,n}^c+\hat\alpha_{p,n}^c-2)
    =O_{\mathbb P}(1),
\]
and hence
\[
    \mathbb P\bigl(
        |\hat\alpha_{1,n}^c+\hat\alpha_{p,n}^c-2|
        \le\log(n)/\sqrt n
    \bigr)\to1.
\]
Thus $\mathbb P(I_n=I^\star)\to1$ in all cases.
The same comparison and the joint convergence above imply
\begin{align}
\label{eq:main_test_centered_limit}
    \tilde T_n-\sqrt n\{\mathcal S(\bfSigma)-\Delta\}
    \xrightarrow{d}\max_{i\in I^\star}Z_i.
\end{align}
Indeed, when the population deviations differ, the larger
determines the empirical maximum with probability tending
to one; when they coincide, the result follows by the
continuous mapping theorem.\\
Now, let $F$ denote the distribution function of
$\max_{i\in I^\star}Z_i$. Positive definiteness of $\mathbf C$
implies that $(Z_i)_{i\in I^\star}$ has an everywhere positive
Gaussian density. Consequently, $F$ is continuous and strictly
increasing, and its $(1-\alpha)$-quantile $q$ is finite and
satisfies $F(q)=1-\alpha$.
Convergence of the simulated critical
value follows by standard arguments (omitted here) and it means formally that
\[
    q_{1-\alpha}^{B_n}\xrightarrow{\mathbb P}q.
\]
If now $\mathcal S(\bfSigma)=\Delta$ (boundary of the hypothesis), then
\eqref{eq:main_test_centered_limit}, convergence of the
critical value and continuity of $F$ yield
\[
    \mathbb P\bigl(\tilde T_n>q_{1-\alpha}^{B_n}\bigr)
    \longrightarrow
    \mathbb P\bigl(\max_{i\in I^\star}Z_i>q\bigr)
    =\alpha.
\]
If $\mathcal S(\bfSigma)<\Delta$, respectively
$\mathcal S(\bfSigma)>\Delta$,
\eqref{eq:main_test_centered_limit} implies
$\tilde T_n\xrightarrow{\mathbb P}-\infty$, respectively
$\tilde T_n\xrightarrow{\mathbb P}+\infty$.
Since the critical value converges to the finite number $q$,
the rejection probabilities converge to zero and one,
respectively.
\end{proof}

\subsection{Proof of Theorem \ref{thm_rmt}} \label{sec_proof_thm_rmt}

In Theorem~\ref{thm_rmt}, quantities $\hat\alpha_{i,n}^c$ defined in \eqref{eq_def_hat_alpha1} and \eqref{eq_def_hat_alphap} correspond to estimation of eigenvalues of $\bfD^{-1/2}\bfSigma\bfD^{-1/2}$ via eigenvalues of $\hat\bfD^{-1/2}\hat\bfSigma\hat\bfD^{-1/2}$. However, it is easier to conduct the theoretical study using the eigenvalues of $\bfD^{-1/2}\hat\bfSigma\bfD^{-1/2}$. We denote the resulting estimators by
\begin{align*}
   \hat\alpha_{i,n} = -\lb-\frac{1-c_n}{\lambda_i(\bfD^{-1/2}\hat\bfSigma\bfD^{-1/2})}+\frac{1}{n}\sum_{j\ne i}\frac{1}{\lambda_j(\bfD^{-1/2}\hat\bfSigma\bfD^{-1/2})-\lambda_i(\bfD^{-1/2}\hat\bfSigma\bfD^{-1/2})}\rb^{-1},
\end{align*}
where $c_n=p/n$. The following lemma ensures validity of this procedure, and its proof is deferred to the end of this section.
\begin{lemma}\label{lem:alpha_hat-alpha_hat_c}
Under Assumptions \ref{ass1_dim}--\ref{ass5_groups}, for $i=1$ and, if $c<1$, also for $i=p$,
    \begin{align*}
        |\hat{\alpha}_{i,n}-\hat{\alpha}_{i,n}^c|=o_\PR(n^{-1/2}).
    \end{align*}
\end{lemma}
In what follows we denote by $\bfLambda$ the diagonal matrix of eigenvalues of  $\bfD^{-1/2}\bfSigma\bfD^{-1/2}$, i.e. $\bfLambda=\operatorname{diag}(\alpha_1,\ldots,\alpha_p)$. Then the singular value decomposition of  matrix $\bfD^{-1/2}\bfGamma$ has a form
  \begin{align*}
    \bfD^{-1/2}\bfGamma=\bfV\bfLambda^{1/2}\bfU^\intercal.
  \end{align*}
  
Furthermore, we define $\mathbf{1}_n$ as the $n$ dimensional vector of  ones. 
By letting $\boldsymbol{\Phi}=\mathbf{I}_n-\frac{1}{n}\mathbf{1}_n\mathbf{1}_n'$, the corrected sample covariance matrix $\bfD^{-1/2}\hat\bfSigma_n\bfD^{-1/2}$ can be written as 
\begin{align}
   \bfD^{-1/2}\hat\bfSigma_n\bfD^{-1/2} =  \frac{1}{n} \bfD^{-1/2}\bfGamma \bfQ \boldsymbol{\Phi} \bfQ^\top \bfGamma^\top\bfD^{-1/2}. 
\end{align}
with eigenvalues $\hat{\lambda}_{n,1}\ge\ldots\ge\hat{\lambda}_{n,p}$. By polar decomposition, these are also the eigenvalues of $n^{-1}\bfB_n^{1/2}\bfQ\boldsymbol\Phi\bfQ^\top\bfB_n^{1/2}$, where $\bfB_n=\bfGamma^\top\bfD^{-1}\bfGamma=\bfU\bfLambda\bfU^\top$. Thus the population eigenvectors in Appendix \ref{sec:app_rmt} are the columns of $\bfU$.
Moreover, for any non-negative definite matrix $\mathbf{A} \in \R^{m\times m}$ with eigenvalues $\lambda_1(\bfA), \ldots, \lambda_m (\bfA)$, we denote its empirical spectral measure by 
\begin{align}
    F^{\bfA} = \frac{1}{m} \sum_{i=1}^m \delta_{\lambda_i(\bfA)},
\end{align}
where $\delta_{\cdot}$ denotes the Dirac measure. 

\begin{proof}[Proof of Theorem \ref{thm_rmt}]
We focus on part \ref{thm_part_a}, since part \ref{thm_part_b} follows the same proof, but with case $k=p$ removed. We begin by showing a similar CLT result for the related expressions $\underline{s}_{n,k}+1/\alpha_k$, where we denote $\underline{s}_{n,k}=-1/\hat\alpha_{k,n}$ for $k \in \{1,p\}$. Then, with Lemma~\ref{lem:alpha_hat-alpha_hat_c}, we immediately recover Theorem~\ref{thm_rmt}. This proof follows the general strategy of \cite{bai:ding:2012}, but uses the general-factor eigenvalue CLT of \cite{zhang:zheng:pan:zhong:2022} in place of the independent spike and bulk blocks assumed in that work.

To deal with $\underline{s}_{n,k}+1/\alpha_{k}$, we define
$\boldsymbol{\Lambda}_{-k}$ as the $(p-1)\times(p-1)$ matrix obtained from $\bfLambda$ by removing the $k$th row and column. In simpler terms, for $k=1$ we remove the first eigenvalue, and for $k=p$ we remove the last one. Similarly, the $p\times (p-1)$ matrix $\bfU_{-k}$ is obtained from $\bfU$ by removing the $k$th column (eigenvector). 
    Let us denote the $p-1$  eigenvalues of $\bfS_{-k}:=n^{-1}\bfLambda_{-k}^{1/2}\bfU_{-k}^\intercal\bfQ\boldsymbol{\Phi}\bfQ^\intercal\bfU_{-k}\bfLambda_{-k}^{1/2}$ by $\hat{\mu}_{n,1}^{(-k)}\ge\hat\mu_{n,2}^{(-k)}\ge\ldots\ge\hat\mu_{n,p-1}^{(-k)}$. 
    First we notice that 
    \begin{align*}
        &\bfS_{-1}=\frac1n
        \begin{pmatrix}
            0&\bfI_{p-1}
        \end{pmatrix}\begin{pmatrix}
            \alpha_1^{1/2}&0\\
            0&\bfLambda_{-1}^{1/2}
        \end{pmatrix}\bfU^\intercal \bfQ\boldsymbol{\Phi}\bfQ^\intercal \bfU
       \begin{pmatrix}
            \alpha_1^{1/2}&0\\
            0&\bfLambda_{-1}^{1/2}
        \end{pmatrix}
        \begin{pmatrix}
            0
            \\
            \bfI_{p-1}
        \end{pmatrix},\\
        &\bfS_{-p}=\frac1n \begin{pmatrix}
            \bfI_{p-1}&0
        \end{pmatrix}\begin{pmatrix}
            \bfLambda_{-p}^{1/2}&0\\
            0&\alpha_p^{1/2}
        \end{pmatrix}\bfU^\intercal \bfQ\boldsymbol{\Phi}\bfQ^\intercal \bfU
       \begin{pmatrix}
            \bfLambda_{-p}^{1/2}&0\\
            0&\alpha_p^{1/2}
        \end{pmatrix}
        \begin{pmatrix}
            \bfI_{p-1}
            \\
            0
        \end{pmatrix}.
    \end{align*}
    The matrices on the r.h.s. of the previous display surrounded by $(0\;\bfI_{p-1})$ and $(\bfI_{p-1}\;0)$, that is, the matrices
    \begingroup\setlength{\arraycolsep}{3pt}
    \begin{align*}
        \frac1n\begin{pmatrix}
            \alpha_1^{1/2}&0\\
            0&\bfLambda_{-1}^{1/2}
        \end{pmatrix}\bfU^\intercal \bfQ\boldsymbol{\Phi}\bfQ^\intercal \bfU
       \begin{pmatrix}
            \alpha_1^{1/2}&0\\
            0&\bfLambda_{-1}^{1/2}
        \end{pmatrix}
        \ \textnormal{and}\  
        \frac1n\begin{pmatrix}
            \bfLambda_{-p}^{1/2}&0\\
            0&\alpha_p^{1/2}
        \end{pmatrix}\bfU^\intercal \bfQ\boldsymbol{\Phi}\bfQ^\intercal \bfU
       \begin{pmatrix}
            \bfLambda_{-p}^{1/2}&0\\
            0&\alpha_p^{1/2}
        \end{pmatrix}
    \end{align*}
    \endgroup
    have the same eigenvalues as $\bfD^{-1/2}\bfSigmahat_n\bfD^{-1/2}$. Then, the interlacing theorem for eigenvalues gives us almost surely
    \begin{align}\label{eq:interlac_eigenval}
        \hat{\lambda}_{n,i}\ge\hat{\mu}^{(-k)}_{n,i}\ge\hat{\lambda}_{n,1+i}, \quad 1 \leq i \leq p-1.
    \end{align}
   
   To decompose $\underline{s}_{n,k}+1/\alpha_k$ into simpler terms, we introduce two intermediate expressions 
    \begin{align}
            \underline{\hat{s}}_{n,k}&=-\frac{1-c_n}{\hat{\lambda}_{n,k}}+\frac{1}{n}\sum_{j=1}^{p-1}\frac{1}{\hat{\mu}^{(-k)}_{n,j}-\hat{\lambda}_{n,k}}, \\
             \underline{\bar{s}}_{n,k}&=-\frac{1-c_n}{\lambda_{n,k}}+\frac{1}{n}\sum_{j=1}^{p-1}\frac{1}{\hat{\mu}^{(-k)}_{n,j}-\lambda_{n,k}},
    \end{align}
    where $\lambda_{n,k}$ is an approximation of $\hat{\lambda}_{n,k}$ of the form
    \begin{align}\label{eq:lambda_n_k}
        \lambda_{n,k}=\phi_{n,k}(\alpha_k):=\alpha_k+c_n\int\frac{\alpha_kt}{\alpha_k-t}dH_{n,k}^{(0)}(t)=\alpha_k+\frac1n\sum_{j\ne k}\frac{\alpha_k\alpha_j}{\alpha_k-\alpha_j}.
    \end{align}
    Here $H_{n,k}^{(0)}:=p^{-1}(\delta_0+\sum_{j\ne k}\delta_{\alpha_j})$ replaces the target spike by zero. The function $\phi_{n,k}$ is obtained from $\phi$ by replacing $(c,H)$ with $(c_n,H_{n,k}^{(0)})$. This removes the pole from the finite-sample centering; no rate of convergence of $H_n$ to $H$ is required.
    Then, we decompose
    \begin{align}
        \underline{s}_{n,k}+1/\alpha_k=(\underline{s}_{n,k}-\underline{\hat{s}}_{n,k})+(\underline{\hat{s}}_{n,k}-\underline{\bar{s}}_{n,k})+(\underline{\bar{s}}_{n,k}+1/\alpha_k):=I_1^{(k)}+I_2^{(k)}+I_3^{(k)},
    \end{align}
    where $I_1^{(k)}, I_2^{(k)}, I_3^{(k)}$ are defined in the obvious way. In the following, we will show that $I_1^{(k)}, I_3^{(k)} = O_{\PR}(1/n)$, while $\sqrt{n} I_2^{(k)}$ satisfies a CLT.

   Recall the definitions $\underline{s}_{n,k}=-1/\hat{\alpha}_{k,n}$, and  
    \begin{align*}
        \hat{\alpha}_{k,n}=-\lb-\frac{1-c_n}{\hat{\lambda}_{n,k}}+\frac{1}{n}\sum_{i\neq k}\frac{1}{\hat{\lambda}_{n,i}-\hat{\lambda}_{n,k}}\rb^{-1}.
    \end{align*}
This gives for the first term
\begin{align}
    &I_1^{(1)}=
        \frac{1}{n}\sum_{j=1}^{p-1}\frac{\hat{\mu}^{(-1)}_{n,j}-\hat{\lambda}_{n,j+1}}{(\hat{\lambda}_{n,j+1}-\hat{\lambda}_{n,1})(\hat{\mu}^{(-1)}_{n,j}-\hat{\lambda}_{n,1})},\\
      & I_1^{(p)}= \frac{1}{n}\sum_{j=1}^{p-1}\frac{\hat{\mu}^{(-p)}_{n,j}-\hat{\lambda}_{n,j}}{(\hat{\lambda}_{n,j}-\hat{\lambda}_{n,p})(\hat{\mu}^{(-p)}_{n,j}-\hat{\lambda}_{n,p})}.
\end{align}
Due to Assumption~\ref{ass4_spikes} and the exact-separation argument underlying Proposition~\ref{thm:sample_outlier_convergence}, the target eigenvalues are uniformly separated from the remaining sample eigenvalues and from the spectrum of $\bfS_{-k}$, with probability tending to one. Indeed, $\phi_{n,k}'$ is bounded away from zero on a fixed neighborhood of $\alpha_k$, whose image is a spectral gap for the model with that spike removed. All eigenvalues are bounded in probability. Using the signs in \eqref{eq:interlac_eigenval}, we obtain
\begin{align}
    &I_1^{(1)}=
        \frac{\mathcal{O}_{\PR}(1)}{n}\lb\sum_{j=1}^{p-1}\hat{\mu}^{(-1)}_{n,j}-\sum_{j=2}^p\hat{\lambda}_{n,j}\rb,\\
       -&I_1^{(p)}= \frac{\mathcal{O}_{\PR}(1)}{n}\lb\sum_{j=1}^{p-1}\hat{\lambda}_{n,j}-\sum_{j=1}^{p-1}\hat{\mu}^{(-p)}_{n,j}\rb.
\end{align}
For the first expression we use l.h.s. of \eqref{eq:interlac_eigenval} and write
\begin{align*}
   0\leq  I_1^{(1)}\leq \frac{\mathcal{O}_{\PR}(1)}{n}\lb\hat{\lambda}_{n,1}-\hat{\lambda}_{n,p}\rb=\mathcal{O}_{\PR}\lb\frac{1}{n}\rb,
\end{align*}
while for the second one, we use  the r.h.s of \eqref{eq:interlac_eigenval} to obtain the same expression as above:
\begin{align*}
   0\leq  -I_1^{(p)}\leq  \frac{\mathcal{O}_{\PR}(1)}{n}\lb\hat{\lambda}_{n,1}-\hat{\lambda}_{n,p}\rb=\mathcal{O}_{\PR}\lb\frac{1}{n}\rb.
\end{align*}

To deal with the third term $I_3^{(k)}$, we first define $\underline{s}_{n,k}^{(0)}(z)$ as the companion Stieltjes transform determined by
\begin{align}
    z=-\frac{1}{\underline{s}_{n,k}^{(0)}(z)}+c_n\int\frac{t}{1+t\underline{s}_{n,k}^{(0)}(z)}dH_{n,k}^{(0)}(t).
\end{align}
Then, plugging $z=\lambda_{n,k}$ and comparing with \eqref{eq:lambda_n_k}, we get, for all sufficiently large $n$,
\begin{align} \label{eq_alpha_formula}
    -\frac{1}{\alpha_k}=\underline{s}_{n,k}^{(0)}(\lambda_{n,k}).
\end{align}
The branch is fixed by $\phi_{n,k}'(\alpha_k)>0$. According to Proposition~\ref{th:steltjes}, we can write
\begin{align}
    \underline{s}_{n,k}^{(0)}(z)=-\frac{1-c_n}{z}+c_n\int\frac{1}{x-z}dF^{c_n,H_{n,k}^{(0)}}(x).
\end{align}
Let $\bfS_{-k}^{(0)}=\operatorname{diag}(0,\bfS_{-k})$, so that
$F^{\bfS_{-k}^{(0)}}=p^{-1}\delta_0+(p-1)p^{-1}F^{\bfS_{-k}}$.
These preparations show that
\begin{align*}
    I_3^{(k)}
    &=-\frac{1-c_n}{\lambda_{n,k}}+\frac1n\sum_{j=1}^{p-1}\frac{1}{\hat\mu_{n,j}^{(-k)}-\lambda_{n,k}}-\underline{s}_{n,k}^{(0)}(\lambda_{n,k})\\
    &=\frac{1}{n\lambda_{n,k}}+c_n\int\frac{d(F^{\bfS_{-k}^{(0)}}-F^{c_n,H_{n,k}^{(0)}})(x)}{x-\lambda_{n,k}}
    =\mathcal O_\PR(n^{-1}).
\end{align*}
For the last equality, the tightness and bias bounds in Theorems 2 and 3 of \cite{najim:yao:2016:proof} apply to a smooth extension of $x\mapsto(x-\lambda_{n,k})^{-1}$ off the fixed spectral gap described above. These extensions have uniformly bounded derivatives. By polar decomposition, the same nonzero eigenvalues arise from a symmetric-factor sample covariance matrix with population spectral distribution $H_{n,k}^{(0)}$. Empirical centering changes its resolvent trace by $\mathcal O_\PR(1)$ on this gap, by rank-one interlacing. This concludes the proof of $I_3^{(k)}=\mathcal O_\PR(n^{-1})$.

Passing to the remaining term $I_2^{(k)}$, we write
\begin{align}
    I_2^{(k)}
    &=\hat{\underline{s}}_{n,k}-\bar{\underline{s}}_{n,k}
    =\frac{(1-c_n)(\hat{\lambda}_{n,k}-\lambda_{n,k})}{\hat{\lambda}_{n,k}\lambda_{n,k}}+\frac{1}{n}\sum_{j=1}^{p-1}\frac{\hat{\lambda}_{n,k}-\lambda_{n,k}}{(\hat{\mu}^{(-k)}_{n,j}-\lambda_{n,k})(\hat{\mu}^{(-k)}_{n,j}-\hat{\lambda}_{n,k})}\\
    &=(\hat{\lambda}_{n,k}-\lambda_{n,k})\lb\frac{(1-c_n)}{\hat{\lambda}_{n,k}\lambda_{n,k}}+\lb c_n-\frac{1}{n}\rb\int\frac{dF^{\bfS_{-k}}(x)}{(x-\lambda_{n,k})(x-\hat{\lambda}_{n,k})}\rb .\label{e1}
\end{align}
The second factor in \eqref{e1} satisfies
\begin{align}
& \frac{(1-c_n)}{\hat{\lambda}_{n,k}\lambda_{n,k}}+\lb c_n-\frac{1}{n}\rb\int\frac{dF^{\bfS_{-k}}(x)}{(x-\lambda_{n,k})(x-\hat{\lambda}_{n,k})} \\
& = (1-c)/\phi^2(\alpha_k)+c\int\frac{dF^{c,H}(x)}{(x-\phi(\alpha_k))^2} + o_{\PR}(1) := \tau_k+ o_{\PR}(1). \label{eq:def_tau}
\end{align} 
For the first factor in \eqref{e1}, we use Theorem 2.2 and Corollary 2.4 of \cite{zhang:zheng:pan:zhong:2022}, with the lower-outlier extension explained after Proposition~\ref{th:clt_eigen}. The corollary covers empirical centering. This gives 
\begin{align}
      n^{1/2} \lb \hat{\lambda}_{n,k}-\lambda_{n,k} \rb_{k\in\{1,p\} }  \cond \mathcal{N}\lb 0, \lb \Omega_{ij}\rb_{1 \leq i,j \leq 2} \rb , \quad n\to\infty, \label{eq_clt}
\end{align}
    where, using the notation $\gamma_4 = \E [q_{11}^4],$
    \begin{align}\label{eq:def_Omega}
        &\Omega_{11}=\alpha_{\max}^2(\phi^\prime(\alpha_{\max}))^2\lim\left[(\gamma_4-3)\sum_{i=1}^p U_{i1}^4\right]+2\alpha_{\max}^2\phi^\prime(\alpha_{\max}),\\
        &\Omega_{22}=\alpha_{\min}^2(\phi^\prime(\alpha_{\min}))^2\lim\left[(\gamma_4-3)\sum_{i=1}^p U_{ip}^4\right]+2\alpha_{\min}^2\phi^\prime(\alpha_{\min}),\\
        &\Omega_{12}=\Omega_{21}=\alpha_{\max}\alpha_{\min}\phi^\prime(\alpha_{\max})\phi^\prime(\alpha_{\min})\lim\left[(\gamma_4-3)\sum_{i=1}^p U_{i1}^2 U_{ip}^2\right].
    \end{align}

In summary, we have shown that 
\begin{align}
\sqrt{n} \lb \underline{s}_{n,k}+1/\alpha_k \rb_{k\in\{1,p\}} = \sqrt{n}(I_2^{(k)})_{k\in\{1,p\}} + \op  .
    \end{align}
    By combining this result with \eqref{e1}, \eqref{eq:def_tau} and \eqref{eq_clt}, we obtain 
    \begin{align*}
        \sqrt{n} \lb \underline{s}_{n,k}+1/\alpha_k \rb_{k\in\{1,p\}} \cond 
 \mathcal{N}_2\lb0,
 \begin{pmatrix}
            \tau_{\max}&0\\
            0&\tau_{\min}
        \end{pmatrix}\boldsymbol{\Omega}\begin{pmatrix}
            \tau_{\max}&0\\
            0&\tau_{\min}
        \end{pmatrix}  \rb, \quad n\to\infty.
    \end{align*}
Then, an application of the multivariate delta method shows
\begin{align*}
 &\sqrt{n} \lb -1/\underline{s}_{n,k}-\alpha_k \rb_{k\in\{1,p\}} \\
 &\quad\cond 
    \mathcal{N}\lb0,
    \left[ \left\{ 
    \begin{pmatrix}
            \alpha_{\max}^2\tau_{\max}&0\\
            0&\alpha_{\min}^2\tau_{\min}
        \end{pmatrix}\boldsymbol{\Omega}\begin{pmatrix}
            \alpha_{\max}^2\tau_{\max}&0\\
            0&\alpha_{\min}^2\tau_{\min}
        \end{pmatrix}
        \right\}_{i,j} \right]_{1 \leq i,j \leq 2} 
        \rb,
\end{align*}
which concludes the proof. 

\end{proof}

We conclude this section by giving the proof of Lemma \ref{lem:alpha_hat-alpha_hat_c}.

\begin{proof}[Proof of Lemma \ref{lem:alpha_hat-alpha_hat_c}]
For convenience we denote the eigenvalues of $\bfD^{-1/2}\hat\bfSigma\bfD^{-1/2}$ by $\hat\lambda_{n,1}\ge\ldots\ge\hat\lambda_{n,p}$, and eigenvalues of $\hat\bfD^{-1/2}\hat\bfSigma\hat\bfD^{-1/2}$ by $\hat\lambda^c_{n,1}\ge\ldots\ge\hat\lambda^c_{n,p}$.
We start by bounding the relative error of each empirical group variance. Write $\bfz_j=\bfD^{-1/2}(\bfy_j-\mu)=\bfA\bfq_j$, where $\bfA=\bfD^{-1/2}\bfGamma$, and let $\bfP_i$ be the $\ell_i\times p$ coordinate selection matrix for group $i$. Then
\begin{align}
    \frac{\hat\sigma_i^2}{\sigma_i^2}-1
    =\frac{1}{n\ell_i}\sum_{j=1}^n\bigl(\|\bfP_i\bfz_j\|^2-\ell_i\bigr)
      -\frac{\|\bfP_i\bar{\bfz}\|^2}{\ell_i}.
\end{align}
The matrix $\bfA^\intercal\bfP_i^\intercal\bfP_i\bfA$ has trace $\ell_i$, bounded norm and hence squared Frobenius norm $\mathcal O(\ell_i)$. The quadratic-form variance formula gives $\Var(\|\bfP_i\bfz_j\|^2)=\mathcal O(\ell_i)$, while $\E\|\bfP_i\bar{\bfz}\|^2=\ell_i/n$. Since $\ell_i\asymp n$ and $L$ is fixed, it follows that
\begin{align}
    \max_{1\le i\le L}\left|\frac{\hat\sigma_i^2}{\sigma_i^2}-1\right|=\mathcal O_\PR(n^{-1}),
    \qquad
    \|\hat\bfD^{-1/2}\bfD^{1/2}-\bfI_p\|=\mathcal O_\PR(n^{-1}).
\end{align}
In particular, the empirical group variances are positive with probability tending to one. Also, $\|\bfD^{-1/2}\hat\bfSigma\bfD^{-1/2}\|=\mathcal O_\PR(1)$, since $\|\bfA\|^2=\alpha_{\max}$ and $\|\bfQ\|=\mathcal O_\PR(\sqrt n)$. Multiplication on both sides by $\hat\bfD^{-1/2}\bfD^{1/2}=\bfI_p+\mathcal O_\PR(n^{-1})$ and Weyl's inequality now yield
\begin{align}
    \max_{1\le j\le p}|\hat\lambda_{n,j}-\hat\lambda^c_{n,j}|
    \le\|\bfD^{-1/2}\hat\bfSigma\bfD^{-1/2}-\hat\bfD^{-1/2}\hat\bfSigma\hat\bfD^{-1/2}\|
    =\mathcal O_\PR(n^{-1}).
\end{align}
Thus the target eigenvalues have the same localization properties under both normalizations. In particular, for $i=1$ and, if $c<1$, also for $i=p$, they are bounded away from zero and uniformly separated from the other eigenvalues, with probability tending to one. Their companion-transform estimates converge to $-1/\alpha_i$, so $\hat\alpha_{i,n}$ and $\hat\alpha^c_{i,n}$ are bounded in probability.

Now we can return to bounding $|\hat\alpha_{i,n}-\hat\alpha_{i,n}^c|$. After some algebra, using the definitions and the triangle inequality, we get
\begin{align}
    |\hat\alpha_{i,n}-\hat\alpha_{i,n}^c|
    &\le\mathcal O_\PR(1)\left(|\hat\lambda_{n,i}-\hat\lambda^c_{n,i}|+\frac1n\sum_{j\ne i}|\hat\lambda_{n,j}-\hat\lambda^c_{n,j}|\right)\\
    &=\mathcal O_\PR(n^{-1})=o_\PR(n^{-1/2}).
\end{align}
\end{proof}

\section*{Funding}
Nina Dörnemann was partially supported by AUFF grant numbers 47221 and 47388. Tim Kutta was partially supported by AUFF grant number 47222. Daria Tieplova was supported by both AUFF grants numbers 47388 and 47222.

\newpage 

\section*{Appendix: Technical Details}
\appendix
\section{Details on the covariance matrix \texorpdfstring{$\mathbf{C}$}{C} in Theorem \ref{thm_rmt} and its estimator \texorpdfstring{$\hat{\mathbf{C}}_n$}{C-hat}} \label{sec_C}
Here we gather all the elements from the proof of Theorem \ref{thm_rmt} used to determine matrix $\bfC$. We display both coordinates for $c<1$; for $c\ge1$, only the terms corresponding to $\alpha_{\max}$ are used. For $k=\min,\max$ we have
\begin{align}
   \tau_k= \frac{1-c}{\phi^2(\alpha_k)}+c\int\frac{dF^{c,H}(x)}{(x-\phi(\alpha_k))^2} 
\end{align}
and 
\begin{align}
            &\Omega_{11}=\alpha_{\max}^2(\phi^\prime(\alpha_{\max}))^2\lim\left[(\gamma_4-3)\sum_{i=1}^p U_{i1}^4\right]+2\alpha_{\max}^2\phi^\prime(\alpha_{\max}),\\
        &\Omega_{22}=\alpha_{\min}^2(\phi^\prime(\alpha_{\min}))^2\lim\left[(\gamma_4-3)\sum_{i=1}^p U_{ip}^4\right]+2\alpha_{\min}^2\phi^\prime(\alpha_{\min}),\\
        &\Omega_{12}=\Omega_{21}=\alpha_{\max}\alpha_{\min}\phi^\prime(\alpha_{\max})\phi^\prime(\alpha_{\min})\lim\left[(\gamma_4-3)\sum_{i=1}^p U_{i1}^2 U_{ip}^2\right].
\end{align}
Then, we define the $2\times 2$ matrix $\mathbf{C}$
 as 
 \begin{align}
    \mathbf{C}=
    \begin{pmatrix}
        \alpha_{\max}^2\tau_{\max}&0\\
        0&\alpha_{\min}^2\tau_{\min}
    \end{pmatrix}
    \boldsymbol{\Omega}
    \begin{pmatrix}
        \alpha_{\max}^2\tau_{\max}&0\\
        0&\alpha_{\min}^2\tau_{\min}
    \end{pmatrix}.
\end{align}
To approximate $\tau_k$ we use the fact that $\hat\lambda_{n,k}^c\conp\phi(\alpha_k)$ and $F^{\hat\bfSigma_n^c}$ converges weakly in probability to $F^{c,H}$. Writing $\hat\lambda_{n,k}^c=\lambda_k(\hat\bfSigma_n^c)$, $\tau_1=\tau_{\max}$ and $\tau_p=\tau_{\min}$, we set
\begin{align}
    \hat\tau_k=\frac{1-c_n}{(\hat\lambda_{n,k}^c)^2}+\frac1n\sum_{i\ne k}\frac{1}{(\hat\lambda_{n,i}^c-\hat\lambda_{n,k}^c)^2},
    \qquad k\in\{1,p\}.
\end{align}
The identity $\tau_k=(\alpha_k^2\phi'(\alpha_k))^{-1}$ follows by differentiating the companion Stieltjes equation. Hence the Gaussian part of the diagonal entry of $\bfC$ corresponding to $\alpha_k$ is $2\alpha_k^4\tau_k$. If $v_k$ is the corresponding population eigenvector of $\bfD^{-1/2}\bfSigma\bfD^{-1/2}$, the fourth joint cumulant of the scores $\eta_k=v_k^\intercal\bfD^{-1/2}(\bfy-\mu)$ satisfies
\begin{align}
    \operatorname{cum}(\eta_k,\eta_k,\eta_l,\eta_l)
    =(\gamma_4-3)\alpha_k\alpha_l\sum_{i=1}^pU_{ik}^2U_{il}^2.
\end{align}
Its limit is exactly the remaining term in the corresponding entry of $\bfC$.

We propose estimating these cumulants using a small reserved sample. Reserve $m_n=\lfloor\sqrt n\rfloor$ observations, indexed by $\mathcal V$, and use the other $N=n-m_n$ observations to compute $\hat\bfD_N$, the unit eigenvectors $\hat v_{N,k}$ of $\hat\bfSigma_N^c$, and $\hat\alpha_{k,N}^c$, $\hat\lambda_{N,k}^c$ and $\hat\tau_{N,k}$ by the preceding formulas with $n$ replaced by $N$. Set
\begin{align}
    \hat r_k=\min\left\{1,\max\left\{N^{-1/4},\frac{1}{\hat\alpha_{k,N}^c\hat\lambda_{N,k}^c\hat\tau_{N,k}}\right\}\right\}.
\end{align}
For the reserved observations, whose mean is denoted by $\bar\bfy_{\mathcal V}$, define
\begin{align}
    \xi_{jk}&=\hat v_{N,k}^\intercal\hat\bfD_N^{-1/2}(\bfy_j-\bar\bfy_{\mathcal V}),
    &\hat h_{kl}&=\frac1{m_n}\sum_{j\in\mathcal V}\xi_{jk}\xi_{jl},\\
    \hat K_{kl}&=\frac1{m_n}\sum_{j\in\mathcal V}\xi_{jk}^2\xi_{jl}^2-\hat h_{kk}\hat h_{ll}-2\hat h_{kl}^2.
\end{align}
Our proposed covariance estimator is
\begin{align}
    \tilde C_{kl}&=2\,\mathbf1_{\{k=l\}}(\hat\alpha_{k,n}^c)^4\hat\tau_k+\frac{\hat K_{kl}}{\hat r_k\hat r_l},
    &\hat{\bfC}_n&=\left[(\tilde C_{kl})_{k,l\in\{1,p\}}\right]_+,
\end{align}
where the coordinates are ordered as $(1,p)$ and $[\,\cdot\,]_+$ replaces negative eigenvalues by zero. For $c\ge1$, only $k=l=1$ is used. On samples where a required inverse or a separated eigenvector is not defined, set $\hat\bfC_n$ to the identity of the appropriate size.

The factors $\hat r_k\hat r_l$ are motivated by the squared eigenvector overlaps $r_k=\alpha_k\phi'(\alpha_k)/\phi(\alpha_k)=1/(\alpha_k\phi(\alpha_k)\tau_k)$; see \cite{pu:zhang:hu:bai:2024}. In this approximation, each empirical loading contributes a factor $\sqrt{r_k}$ to its population score, while the remaining loading contributes negligibly to fourth cumulants. We present this construction as a proposed estimator, without asserting its consistency in full generality; Theorem \ref{thm_main_test_sph} is conditional on the stated covariance-consistency assumption.

If $\gamma_4=3$, or if $\sum_iU_{ik}^4\to0$ for each retained spike, all fourth-cumulant terms vanish. In these cases one may use the simpler full-sample estimator with diagonal entries $2(\hat\alpha_{k,n}^c)^4\hat\tau_k$ and zero off-diagonal entries, again projected onto the positive semidefinite cone if necessary.

\section{Background results}\label{sec:app_rmt}

In this section we collect several known facts concerning the behaviour of the eigenvalues of sample covariance matrix that we used in the proof of Theorem~\ref{thm_rmt}. 
Assume we have $B_n$, a $p\times p$ symmetric non-negative matrix with uniformly bounded norm, whose empirical spectral distribution $H_n=F^{B_n}$ converges weakly to a probability distribution $H$. Matrix $X$ is of size $p\times n$; its elements are i.i.d. centered random variables with unit variance and finite fourth moment $\gamma_4$. Dimensions are such that $p/n=c_n\rightarrow c\in(0,\infty)$. We consider a sample covariance
\[
S_n=\frac1n B_n^{1/2}X\Phi_n X^\intercal B_n^{1/2},
\]
where $\Phi_n=I_n$ or $\Phi_n=I_n-n^{-1}\mathbf1_n\mathbf1_n^\intercal$, and here $\mathbf1_n$ is the vector of ones.

\begin{proposition}\label{th:steltjes}
    The empirical spectral distribution $F^{S_n}$ of $S_n$ converges almost surely to a non-random probability distribution $F^{c,H}$, whose Stieltjes transform $s(z)$ is the unique Stieltjes-transform solution for $z\in\mathbb C^+$ of the equation
\begin{equation}\label{eq:prelim_stieltjes}
s(z)=\int \frac{1}{t(1-c-czs(z))-z}\,dH(t).
\end{equation}
\end{proposition}

In order to describe the localization of the eigenvalues of the sample covariance matrix we introduce the adjacent matrix 
\[
\underline{ S}_n=\frac1n\Phi_n X^\intercal B_nX\Phi_n
\]
Its nonzero eigenvalues coincide with those of $S_n$, so its empirical spectral distribution $F^{\underline{S}_n}$ satisfies $F^{\underline{S}_n}(x)=\frac{n-p}{n}\delta_0
(x)+\frac{p}{n}F^{S_n}(x)$. It follows that its limiting distribution, denoted by $\underline{F}^{c,H}$ can be found immediately as $\underline{F}^{c,H}=(1-c)\delta_0
(x)+cF^{c,H}(x)$, as well as its Stieltjes transform 
\begin{align}
    \underline{s}(z)=-\frac{1-c}{z}+c\int\frac{1}{x-z}dF^{c,H}(x).
\end{align}
From Proposition~\ref{th:steltjes} one can derive that $\underline{s}(z)$ is the unique Stieltjes-transform solution of the equation
\begin{align}
    z=-\frac{1}{\underline{s}(z)}+\int\frac{ct}{1+t\underline{s}(z)}dH(t)
\end{align}
In \cite{bai:yao:2012}
 the detailed study of the connection between sample eigenvalues of $S_n$, support $\underline{F}^{c,H}$ and support of $H$ was conducted. We cite here some of the results that are relevant to our work. We recall function $\phi(\lambda)$ defined in \eqref{eq:def_phi} as 
\begin{align}
    \phi(\lambda)=\lambda+c\int\frac{\lambda t}{\lambda-t}dH(t)
\end{align}
which plays a significant role in characterizing supercritical eigenvalues.
\begin{proposition}[Convergence of sample outlier eigenvalues]
\label{thm:sample_outlier_convergence}
Under the assumptions described above, let $\alpha_k\notin\operatorname{supp}(H)$
be a fixed supercritical eigenvalue of multiplicity $m_k$, uniformly separated
from the remaining population eigenvalues. For a lower extreme outlier assume also $c<1$ and $\inf_n\lambda_p(B_n)>0$.
Assume that $\alpha_k$ has descending population ranks
$$
    \nu_k+1,\ldots,\nu_k+m_k.
$$
Then the corresponding $m_k$ consecutive sample eigenvalues satisfy
\begin{align}
    \hat\lambda_{\nu_k+j,n}
    \xrightarrow{\mathrm{a.s.}}
    \phi(\alpha_k),
    \qquad j=1,\ldots,m_k.
    \label{eq:sample_outlier_limit}
\end{align}
\end{proposition}
To characterize all other eigenvalues we use the next Proposition.
\begin{proposition}\label{th:eigen_in_spec}
Under the conditions stated above, there exists a deterministic constant $a_+<\infty$ such that, almost surely for all sufficiently large $n$,
\[
    \operatorname{spec}(S_n)\subset[0,a_+].
\]
If, in addition, $c<1$ and $\operatorname{spec}(B_n)\subset[b_-,b_+]$ for constants $0<b_-<b_+<\infty$, then there is also $a_->0$ such that, almost surely for all sufficiently large $n$,
\[
    \operatorname{spec}(S_n)\subset[a_-,a_+].
\]
\end{proposition}
These bounds follow from the standard extreme-singular-value bounds for $X/\sqrt n$; empirical centering has the same bounds when $c<1$; see \cite{bai:silverstein:2010}.

To formulate the CLT for sample eigenvalues, obtained in \cite{zhang:zheng:pan:zhong:2022}, we label the $K$ selected fixed, simple supercritical eigenvalues of $B_n$ as $\alpha_1>\ldots>\alpha_K>0$, uniformly separated from the remaining population spectrum and outside $\operatorname{supp}(H)$. Their sample counterparts are denoted by $\hat\lambda_1,\ldots,\hat\lambda_K$; their ranks need not be consecutive. Lower extreme outliers are considered when $c<1$ and $\inf_n\lambda_p(B_n)>0$.
\begin{assumption}\label{ass:variance}
Assume that for each $i=1,\ldots,K$ the following limits exist:
\begin{align}
\sigma_i^2
&:=
\lim_{p\to\infty}
(\gamma_4-3)
\frac{\alpha_i^2\{\phi'(\alpha_i)\}^2}{\phi(\alpha_i)^2}
\sum_{j=1}^{p}u_{ij}^4
+
\frac{2\alpha_i^2\phi'(\alpha_i)}{\phi(\alpha_i)^2},
\label{eq:sigmaii}
\\
\sigma_{ij}
&:=
\lim_{p\to\infty}
(\gamma_4-3)
\frac{\alpha_i\alpha_j
\phi'(\alpha_i)\phi'(\alpha_j)}
{\phi(\alpha_i)\phi(\alpha_j)}
\sum_{k=1}^{p}
u_{ik}^2u_{jk}^2,
\qquad i\neq j,
\label{eq:sigmaij}
\end{align}
where $\gamma_4=\E X_{11}^4$ and
$u_i=(u_{i1},\ldots,u_{ip})^\top$ denotes the population eigenvector
associated with the spike $\alpha_i$.
\end{assumption}
\begin{proposition}\label{th:clt_eigen}
Suppose assumptions from above hold and let
\[
\lambda_i=\phi_{n,i}(\alpha_i),\qquad i=1,\ldots,K,
\]
where $\phi_{n,i}$ is obtained from $\phi$ by replacing $(c,H)$ with $(c_n,H_{n,i}^{(0)})$, and $H_{n,i}^{(0)}=H_n+p^{-1}(\delta_0-\delta_{\alpha_i})$.

Define
\[
\Lambda_n
=
\left(
\sqrt{n}\frac{\hat\lambda_1-\lambda_1}{\lambda_1},
\ldots,
\sqrt{n}\frac{\hat\lambda_K-\lambda_K}{\lambda_K}
\right).
\]
Then
\[
\Lambda_n
\xrightarrow{d}
N(0,\Sigma),
\]
where $\Sigma=(\Sigma_{ij})_{1\le i,j\le K}$ is the covariance matrix
given by
\[
\Sigma_{ij}
=
\begin{cases}
\sigma_i^2,& i=j,\\
\sigma_{ij},& i\neq j,
\end{cases}
\]
with $\sigma_i^2$ and $\sigma_{ij}$ defined in
Assumption~\ref{ass:variance}.
\end{proposition}
The empirical-centering statement is Corollary 2.4 of \cite{zhang:zheng:pan:zhong:2022}. Their Theorem 2.2 is stated for upper outliers. The same proof gives the lower-outlier and mixed upper/lower statements used here: exact separation provides fixed neighborhoods of the outlier limits, away from zero and the remaining spectrum, on which the resolvents are uniformly bounded. The determinant expansion and the joint quadratic-form calculation then apply with the same covariance formulas. Replacing the other finitely many selected spikes by zero in the deterministic centering changes $\lambda_i$ by only $\mathcal O(n^{-1})$.

\bibliographystyle{plainnat}
\bibliography{references,proof_references}

\end{document}